\documentclass[12pt,reqno]{amsart}
\usepackage[T1]{fontenc}
\usepackage{lmodern}
\usepackage{amsmath,amssymb,amsthm,mathtools}
\usepackage{xcolor,microtype,needspace}
\usepackage[unicode,colorlinks=true,linkcolor=blue!45!black,citecolor=blue!45!black,urlcolor=blue!45!black]{hyperref}
\makeatletter
\renewcommand\part{\@startsection{part}{0}%
  \z@{\linespacing\@plus\linespacing}{.5\linespacing}%
  {\normalfont\bfseries\centering}}
\makeatother
\newcommand{\R}{\mathbb R}
\newcommand{\C}{\mathbb C}

\newcommand{\E}{\mathbb E}

\DeclareMathOperator{\tr}{tr}

\DeclareMathOperator{\Cov}{Cov}
\newtheorem{theorem}{Theorem}[section]
\newtheorem{lemma}[theorem]{Lemma}
\newtheorem{proposition}[theorem]{Proposition}
\newtheorem{corollary}[theorem]{Corollary}
\theoremstyle{definition}

\theoremstyle{remark}
\newtheorem{remark}[theorem]{Remark}
\numberwithin{equation}{section}
\title[Riesz kernels and Jordan rigidity]{Riesz kernels of hyperbolic polynomials:\\ positivity, admissible exponents and Jordan rigidity}
\author{Dongsheng Wei}
\address{Independent researcher}
\email{dongshengwei2025@icloud.com}
\date{September 29, 2026}
\hypersetup{pdftitle={Riesz kernels of hyperbolic polynomials: positivity, admissible exponents and Jordan rigidity},pdfauthor={Dongsheng Wei}}
\DeclareMathOperator{\rank}{rank}
\begin{document}
\begin{abstract}
Scott and Sokal asked whether every homogeneous polynomial with the half-plane property has a completely monotone negative power. We answer this question affirmatively and prove the Riesz-positivity conjecture of Micha{\l}ek, Sturmfels, Uhler and Zwiernik, restated by Kozhasov, Micha{\l}ek and Sturmfels: in $n$ variables, every exponent $\alpha\ge4096n^2$ is admissible, independently of the degree and coefficients. For complete hyperbolic polynomials the Riesz density is strictly log-concave, with relative Gaussian error at most $512n^2/\alpha$ and explicit curvature bounds. We characterize admissible exponents by a common spectral Dirichlet law, prove $n\le m+\alpha m(m-1)$ with its equality case, and obtain the sharp degree-dependent gap $0<\alpha<1/(2(m-1))$ whenever the degree-$m$ polynomial has a nonlinear irreducible factor. A nonnegative fourth-order defect of $-\log p$ vanishes at one point precisely for products of positive integer powers of Euclidean Jordan determinants. We classify the corresponding logarithmic Monge--Amp\`ere equation and answer the question of Etingof, Kazhdan and Polishchuk about polynomial multiplicative Legendre transforms within irreducible complete hyperbolic polynomials; the general question has counterexamples, the Clifford quartics of Kogiso and Sato. The classification extends to reducible polynomials under a boundary-visibility hypothesis, and asymptotic common-power formulas for the Riesz densities force exact Jordan formulas.
\end{abstract}
\maketitle
\setcounter{tocdepth}{1}
\noindent\begin{minipage}{\textwidth}
\tableofcontents
\end{minipage}
\par
\medskip
\section{Introduction}\label{sec:intro}
Let $p$ be a homogeneous real polynomial, positive on an open convex cone $C\subset\R^n$. We ask for which exponents $\alpha>0$ there is a positive measure $\mu_\alpha$ on the dual cone such that
\begin{equation}\label{eq:intro-laplace}
 p(x)^{-\alpha}=\int_{C^*}e^{-\langle x,y\rangle}\,d\mu_\alpha(y),
 \qquad x\in C.
\end{equation}
By the Bernstein--Hausdorff--Widder--Choquet theorem, this is equivalent to complete monotonicity of $p^{-\alpha}$ on $C$ \cite{Choquet}, \cite[Theorem~2.2]{SS}. When $p$ is hyperbolic, G\aa rding's construction of the Riesz kernel gives the inverse Laplace transform \cite[Theorem~3.1]{Garding1951}; Atiyah, Bott and G\aa rding developed the associated distributions further in \cite[Lemma~4.15]{ABG}. The question is whether this inverse transform is positive.

Scott and Sokal proved that a polynomial with a completely monotone negative power has the half-plane property \cite[Corollary~2.3]{SS}, and asked whether a homogeneous polynomial with that property could have no completely monotone negative power \cite[Remark~3 following Corollary~2.3]{SS}. Micha{\l}ek, Sturmfels, Uhler and Zwiernik conjectured that every hyperbolic polynomial gives rise to a statistical exponential family with an underlying measure. For complete hyperbolic polynomials, they formulated this as nonnegativity of the Riesz kernel for all sufficiently large exponents \cite[Conjecture~3.5]{MSUZ}. Kozhasov, Micha{\l}ek and Sturmfels restated this conjecture in terms of integral certificates of positivity \cite[Conjecture~4.10]{KMS}. We prove it with a bound depending only on the dimension. Corollary~\ref{cor:exponential-family} describes the resulting exponential family and its moments.

Several classes admit more precise answers. For the determinant on the cone of positive definite real symmetric $r\times r$ matrices, the admissible exponents form the Wallach set
\begin{equation}\label{eq:intro-wallach}
 \left\{0,\frac12,\ldots,\frac{r-1}{2}\right\}
       \cup\left(\frac{r-1}{2},\infty\right).
\end{equation}
This follows from the positivity theory of Riesz distributions on symmetric cones; see \cite[Theorem~1]{Gindikin} and \cite[Theorem~1.3(a)]{SS}. Products of linear forms positive on $C$ have completely monotone negative powers for every positive exponent, by the one-dimensional gamma integral. For elementary symmetric polynomials, Kozhasov, Micha{\l}ek and Sturmfels proved complete monotonicity for all sufficiently large exponents \cite[Theorem~6.4]{KMS}. These representations also connect complete monotonicity with the canonical forms of positive geometries, as studied by Mazzucchelli and Raman \cite{MR}. These positivity results rely on special structure of the polynomial: a definite determinantal representation, a factorization into linear forms, or the specific form of the elementary symmetric polynomials. Such structure is not available for a general hyperbolic polynomial.

Our main result is quantitative. Let $p$ be a homogeneous real polynomial, positive on a hyperbolicity cone $C\subset\R^n$. Write $K=C^*$ for the closed dual cone and
\begin{equation}\label{eq:a-hessian}
 H_p(x)=-D^2\log p(x),\qquad x\in C.
\end{equation}
We call $p$ complete if $\overline C$ contains no line. In this case $H_p$ is positive definite, and the inverse Laplace transform of $p^{-\alpha}$ is a function for the exponents considered below. We denote it by $q_\alpha$.

\Needspace{0.70\textheight}
\begin{theorem}\label{thm:main}
Let $p$ be a complete homogeneous hyperbolic polynomial in $n\geq2$ variables, positive on its hyperbolicity cone $C$, and let $K=C^*$. Assume
\begin{equation}\label{ineq:a-threshold}
 \alpha\geq4096n^2.
\end{equation}
For every $y\in\operatorname{int}K$, there is a unique point $e=e_{\alpha,y}\in C$ such that
\begin{equation}\label{eq:a-saddle}
 y=\alpha\nabla\log p(e).
\end{equation}
Define
\begin{equation}\label{eq:a-Q}
 Q_\alpha(y)=\frac{e^{\langle e,y\rangle}}
 {(2\pi\alpha)^{n/2}p(e)^\alpha\sqrt{\det H_p(e)}},
 \qquad d=\frac{512n^2}{\alpha}.
\end{equation}
Then $q_\alpha$ is a real $C^2$ function supported in $K$, nonnegative everywhere and positive on its interior. Its Gaussian comparison is
\begin{equation}\label{ineq:a-comparison}
 \left|\frac{q_\alpha(y)}{Q_\alpha(y)}-1\right|\leq d.
\end{equation}
Its logarithmic curvature satisfies
\begin{equation}\label{ineq:a-curvature}
 \frac{1-2d}{\alpha}H_p(e)^{-1}
 \preceq-D^2\log q_\alpha(y)
 \preceq\frac{1+4d}{\alpha}H_p(e)^{-1}.
\end{equation}
In particular $q_\alpha$ is strictly log-concave on $\operatorname{int}K$, and $p^{-\alpha}$ is completely monotone on $C$.
\end{theorem}

The bound is uniform over complete hyperbolic polynomials of a given dimension. It yields the same complete-monotonicity conclusion for an arbitrary hyperbolic polynomial after quotienting by the lineality space of its cone. In particular, it answers a question of Scott and Sokal; the half-plane formulation is recorded in Corollary~\ref{cor:half-plane-cm}. The strict log-concavity also gives a concave root of the kernel, since $q_\alpha$ is homogeneous of degree $m\alpha-n$, where $m=\deg p$.

The proof uses the real roots of the restrictions $t\mapsto p(e+tv)$ to normalize the inverse Laplace integral at its saddle. In the Hessian coordinates, the sum of the squares of the root parameters is $\|z\|^2$, so a radial bound for the modulus holds on the whole contour and contains no degree factor. The leading phase is a cubic polynomial whose Gaussian second moment is $O(n^2)$; integrating this polynomial before estimating it is the step that gives an error of order $n^2/\alpha$. The same calculation with one or two additional linear factors controls the derivatives needed for strict log-concavity.

The quadratic dependence on $n$ in the Gaussian comparison is already attained by the Lorentz polynomial $p(x)=x_0^2-x_1^2-\cdots-x_{n-1}^2$. At $y=2\alpha(1,0,\ldots,0)$, the absolute relative error equals $(3n^2-6n+4)/(24\alpha)+O_n(\alpha^{-2})$ as $\alpha\to\infty$ for fixed $n$ (Section~\ref{subsec:lorentz}), although $p^{-\alpha}$ is completely monotone on its hyperbolicity cone for every $\alpha\ge(n-2)/2$ \cite[Theorem~1.9(b)]{SS}.

The admissible exponents also carry algebraic information about $p$. Define
\begin{equation}\label{eq:n-3-1}
G(p)=\{\alpha\ge0:p^{-\alpha}\text{ is completely monotone on }C\}.
\end{equation}
Let $m=\deg p$ and fix $e\in C$. The characteristic roots $\lambda_j^e(v)$ are determined by $p(te-v)=p(e)\prod_j(t-\lambda_j^e(v))$. The compact base of the closed dual cone is $S_e=\{u\in K:u(e)=1\}$.

\begin{theorem}[Spectral probability and the exponent gap]\label{thm:intro-exponents}
Let $p$ be complete, with $n,m\ge2$. For $\alpha>0$, membership in $G(p)$ is equivalent to the existence of a random linear functional $U\in S_e$ whose law satisfies
\[
 U(v)\stackrel d=\sum_{j=1}^m W_j\lambda_j^e(v)
 \quad(v\in\R^n),\qquad W\sim\operatorname{Dir}(\alpha,\ldots,\alpha).
\]
This law is unique and its support has convex hull $S_e$. Every positive $\alpha\in G(p)$ satisfies
\begin{equation}\label{eq:intro-dimension}
 n\le m+\alpha m(m-1).
\end{equation}
If a nonlinear real irreducible factor of $p$ has multiplicity $a$, then
\begin{equation}\label{eq:intro-gap}
 (m-a)\alpha\ge\frac12.
\end{equation}
Consequently, unless $p$ is a product of real linear forms, $G(p)\cap(0,1/(2(m-1)))=\varnothing$. This degree-dependent gap is sharp for every $m\ge2$.
\end{theorem}

Theorem~\ref{thm:n-3-1} identifies the law as the angular part of a tilted Riesz measure; Theorems~\ref{thm:n-5-1} and~\ref{thm:n-5-3} give the two bounds. Equality in \eqref{eq:intro-dimension} occurs exactly for a simple Euclidean Jordan determinant at the first positive exponent in its Wallach set. The gap follows by comparing the mass of a Dirichlet average near a nonlinear characteristic root with the time that a linear function can spend near a curved graph.

The equality case leads to a local criterion for Jordan structure. At $e\in C$, define a commutative product by $H_e(u\circ v,w)=\tfrac12D^3\log p(e)[u,v,w]$ and its fourth-order defect by
\begin{equation}\label{eq:intro-defect}
 \mathcal D_e(v)=\sum_{j=1}^m\lambda_j^e(v)^4-H_e(v\circ v,v\circ v).
\end{equation}
Here $H_e=-D^2\log p(e)$. Two-dimensional determinantal representations imply $\mathcal D_e\ge0$.

\begin{theorem}[One-point rigidity]\label{thm:intro-rigidity}
Let $p$ be complete. The following conditions are equivalent: $\mathcal D_e$ vanishes identically at some $e\in C$; $H_e\overline C=K$ at some $e\in C$; and, after an invertible real linear change of coordinates,
\begin{equation}\label{eq:intro-jordan}
 p=c\prod_{a=1}^s\Delta_{J_a}^{k_a},\qquad c>0,\quad k_a\in\mathbb Z_{>0},
\end{equation}
where $\R^n=\bigoplus_aJ_a$, the $J_a$ are simple Euclidean Jordan algebras, and each determinant uses only its own variable group. If these conditions hold at one point, they hold at every point of $C$.
\end{theorem}

Theorem~\ref{thm:n-4-1} includes the equivalent condition that the cubic tensor of $-\log p$ has zero Levi--Civita covariant derivative at one point. The proof reconstructs the Jordan product from square and cube spectra and an eighth-order identity, then recovers $p$ by Newton identities. This is related to the classification of self-scaled barriers by Hauser and G\"uler \cite{HauserGuler} and to Hildebrand's description of Hessian potentials with parallel third derivatives \cite{Hildebrand}; hyperbolicity makes a condition at one point sufficient here.

The same defect measures the failure of the normalized Hessian determinant to be constant. Set $\Psi=\log\det H+(2n/m)\log p$. Proposition~\ref{thm:n-6-1} expresses $\operatorname{tr}(H^{-1}D^2\Psi)$ as a sum of two nonnegative terms, twice the Gaussian average of $\mathcal D$ and one half of the squared gradient of $\Psi$.

\begin{theorem}[The Hessian determinant equation]\label{thm:intro-ma}
For a complete hyperbolic polynomial $p$, the following are equivalent: $\Psi$ has a local maximum in $C$; $\Psi$ is constant on $C$; and
\[
 \det(-D^2\log p)=c_0p^{-2n/m}\qquad\text{on }C
\]
for some $c_0>0$. These conditions hold precisely for the products \eqref{eq:intro-jordan} satisfying $k_a\rank(J_a)/\dim J_a=m/n$ for every $a$.
\end{theorem}

These products are called balanced. The proof is given in Theorem~\ref{thm:n-6-2}. Fox studied the homogeneous Hessian equation and its relation to affine spheres \cite{Fox}; the classification above is within complete hyperbolic polynomials.

The logarithmic gradient $\tau_p=D\log p$ maps $C$ diffeomorphically onto $C^\vee=\operatorname{int}K$. Its multiplicative Legendre transform is $p_*(y)=p(\tau_p^{-1}(y))^{-1}$. Working over an algebraically closed field of characteristic zero, Etingof, Kazhdan and Polishchuk asked whether polynomiality of this transform forces $p$ to be a relative invariant of a prehomogeneous vector space \cite[Section~3.4, Question~1]{EKP}. Their cubic classification \cite[Theorem~3.10]{EKP}, with the proof completed by Chaput and Sabatino \cite{ChaputSabatino}, gives an affirmative answer in degree three. Kogiso and Sato constructed Clifford quartic forms with polynomial multiplicative Legendre transforms that are not relative invariants of prehomogeneous vector spaces, giving counterexamples to the general question \cite[Theorems~2.14 and~3.2(1)]{KogisoSato}. These counterexamples are absolutely irreducible \cite[Theorem~3.2(3)]{KogisoSato}.

\begin{theorem}[Polynomial multiplicative duality]\label{thm:intro-duality}
Let $p$ be an irreducible complete real hyperbolic polynomial. Its multiplicative Legendre transform is polynomial if and only if, up to an invertible real linear transformation and a positive factor, $p$ is a simple Euclidean Jordan determinant.

More generally, suppose $p$ is complete and every real irreducible factor vanishes on a nonempty relatively open smooth hypersurface piece of $\partial C$ on which all other factors are nonzero. Then $p_*$ is polynomial if and only if $p$ has the Jordan product form \eqref{eq:intro-jordan}.
\end{theorem}

Theorems~\ref{thm:n-7-1} and~\ref{thm:n-7-4} prove these assertions. The complexification of a simple Euclidean Jordan determinant is a relative invariant of the prehomogeneous action of its complexified structure group on $V\otimes_{\R}\C$ \cite[Chapter~VIII]{FK}; thus Theorem~\ref{thm:intro-duality} gives an affirmative answer to their question for the complexifications of irreducible complete hyperbolic polynomials. The transform $p_*$ also governs the exponential part of the Riesz-kernel asymptotics. Theorem~\ref{thm:n-8-1} shows that a pointwise relative approximation $q_{\alpha_j}(\alpha_j y)\sim b_j\Phi(y)^{s_j}$ along a sequence $\alpha_j\to\infty$, with one fixed positive function $\Phi$, forces the balanced Jordan form. The cone gamma integral then gives an exact power formula for all sufficiently large exponents.

The exact positivity threshold remains open in general for elementary symmetric polynomials. With
\[
 E_{r,n}(x)=\sum_{1\le i_1<\cdots<i_r\le n}x_{i_1}\cdots x_{i_r},
 \qquad 2\le r\le n,
\]
Scott and Sokal conjectured that $E_{r,n}^{-\alpha}$ is completely monotone on the positive orthant precisely for the following exponents \cite[Conjecture~1.11]{SS}:
\begin{equation}\label{eq:intro-elementary-conjecture}
 \alpha=0\quad\text{or}\quad\alpha\ge\frac{n-r}{2}
.
\end{equation}
The necessity of this condition is known \cite[Theorem~6.6]{KMS}. Theorem~\ref{thm:main} gives a sufficient bound depending only on $n$; the question is whether it can be lowered to $(n-r)/2$. The Lorentz calculation in Section~\ref{subsec:lorentz} concerns the $n^2/\alpha$ error in Gaussian approximation and allows comparison with the exact positivity range in the quadratic case.

Part~I develops the positive Riesz measures and their admissible exponents. Sections~\ref{sec:background} and~\ref{sec:proof} contain the analytic estimates and the proof of Theorem~\ref{thm:main}, followed by its consequences and the specialized V\'amos quartic. The spectral tools then lead to the probability law and the sharp gap. Part~II develops the spectral metric and its Jordan rigidity, then applies them to cumulants, Hessian determinants, polynomial duality, and asymptotic kernel formulas. The final examples and appendix describe the role of the hypotheses in the duality statements. Regularization on general proper cones and operator-valued extensions will be treated in subsequent work.

\part{Positivity and admissible exponents}
\section{Hyperbolicity and Gaussian moments}\label{sec:background}
\subsection{Hyperbolic polynomials and their cones}
A homogeneous real polynomial $p$ of degree $m\ge1$ is hyperbolic with respect to $e$ if $p(e)>0$ and $t\mapsto p(v+te)$ has only real zeros for every $v\in\R^n$. Its hyperbolicity cone is the connected component $C$ of $\{p>0\}$ containing $e$. G\aa rding's theorem says that $C$ is convex and that $p$ is hyperbolic with respect to every point of $C$ \cite{Garding1959}. Hyperbolicity cones also underlie the barrier functions studied by G\"uler \cite{Guler} and the optimization framework of Renegar \cite{Renegar}.

We fix an inner product and the corresponding Lebesgue measure. The dual cone is
\[
 C^*=\{y\in\R^n:\langle x,y\rangle\ge0\text{ for all }x\in C\}.
\]
The order $A\preceq B$ for symmetric matrices means that $B-A$ is positive semidefinite. We use $D_v$ for differentiation in the direction $v$, and $D^k f(x)[v_1,\ldots,v_k]$ for the $k$th differential.

For a real homogeneous polynomial positive on the orthant, the half-plane property means nonvanishing when all variables have positive real parts. By homogeneity, this is equivalent to hyperbolicity with respect to every positive vector. The relation between this property and matroids is developed in \cite{COSW}; further background on stable polynomials is given in \cite{Pemantle}. We use the right-half-plane convention throughout.

\subsection{Complete monotonicity and Riesz kernels}
A smooth function $f:C\to\R$ is completely monotone if
\begin{equation}\label{ineq:pre-cm}
 (-1)^kD_{v_1}\cdots D_{v_k}f(x)\ge0
 \qquad(x,v_1,\ldots,v_k\in C,\ k\ge0).
\end{equation}
The cone version of the Bernstein--Hausdorff--Widder--Choquet theorem states that this holds if and only if
\[
 f(x)=\int_{C^*}e^{-\langle x,y\rangle}\,d\mu(y)
\]
for a positive measure $\mu$, with the integral finite for $x\in C$ \cite[Theorem~2.2]{SS}; the representing measure is unique, since the Laplace transform is injective. We refer to Choquet \cite{Choquet} for the general cone formulation.

A hyperbolic polynomial positive on $C$ is zero-free on the tube $T_C=C+i\R^n$. This tube is convex, so $p$ has a holomorphic logarithm there whose restriction to $C$ is real. All complex powers of $p$ below use this logarithm. For complete $p$ and $\alpha>n$, its Riesz kernel is given by
\begin{equation}\label{eq:a-inverse}
 q_\alpha(y)=\frac{1}{(2\pi)^n}\int_{\R^n}
       p(x+i\xi)^{-\alpha}e^{\langle x+i\xi,y\rangle}\,d\xi,
 \qquad x\in C.
\end{equation}
The integral is independent of $x$, is supported on $K=C^*$, and has Laplace transform $p^{-\alpha}$ \cite[Theorem~3.1]{Garding1951}; see also \cite[Theorem~3.4]{MSUZ}. In the range of Theorem~\ref{thm:main}, the bounds below give absolute convergence with the moments needed to differentiate twice.

\subsection{Gaussian cubic moments}
The next lemma controls the cubic term of the phase in \eqref{eq:a-inverse}. Keeping this term as a polynomial, instead of bounding it pointwise, is what makes the error of order $n^2/\alpha$.

\begin{lemma}\label{lem:cubic}
Let \(n\ge2\), and let \(P(z)=T[z,z,z]\) be a real cubic with symmetric
coefficient tensor \(T\). Suppose
\begin{equation}\label{ineq:pre-cubic-hyp}
 |P(z)|\le\kappa\|z\|^3\qquad(z\in\mathbb R^n),\qquad \kappa\ge0.
\end{equation}
Put \(c=9\kappa/2\), and let \(Z_0\) be a standard Gaussian vector.
For every unit vector \(w\),
\begin{align}
 \mathbb E P(Z_0)^2&\le15c^2n^2,\label{ineq:pre-cubic-six}\\
 \mathbb E[\langle w,Z_0\rangle^2P(Z_0)^2]
 &\le69c^2n^2.\label{ineq:pre-cubic-eight}
\end{align}
\end{lemma}

\begin{proof}
Polarization gives the exact identity
\[
 T(u,v,w)=\frac1{48}\sum_{\epsilon\in\{-1,1\}^3}
 \epsilon_1\epsilon_2\epsilon_3
 P(\epsilon_1u+\epsilon_2v+\epsilon_3w).
\]
For unit vectors its eight summands have magnitude at most
\(27\kappa\), giving
\begin{equation}\label{ineq:pre-polar}
 |T(u,v,w)|\le\frac{8\cdot27}{48}\kappa=c.
\end{equation}
For a unit \(u\), let \(T_u\) be the matrix with bilinear form
\(T(u,\cdot,\cdot)\), and put \(a_i=\sum_jT_{ijj}\). Then
\begin{equation}\label{ineq:pre-tensor}
 \|T\|_{\mathrm F}^2=\sum_i\|T_{e_i}\|_{\mathrm F}^2\le n^2c^2,
 \qquad \|a\|=\sup_{\|u\|=1}|\operatorname{tr}T_u|\le nc.
\end{equation}
The Gaussian product formula groups the pairings of the six factors
into six pairings joining every index in the first tensor to an index
in the second, and nine pairings with an internal pair in each tensor.
Consequently,
\[
 \mathbb E P(Z_0)^2=6\|T\|_{\mathrm F}^2+9\|a\|^2.
\]
This proves the first estimate.

For the second estimate, Gaussian integration by parts in the direction $w$ yields
\[
 \mathbb E[\langle w,Z_0\rangle^2P^2]
 =\mathbb EP^2+\mathbb E D_w^2(P^2).
\]
Here \(D_wP=3T(w,z,z)\) and \(D_w^2P=6T(w,w,z)\). The quadratic
Gaussian product formula and the linear-cubic product formula give
\begin{align*}
 \mathbb E[\langle w,Z_0\rangle^2P^2]
 ={}&\mathbb EP^2+18\{2\|T_w\|_{\mathrm F}^2+
          (\operatorname{tr}T_w)^2\}\\
 &+36\,a\cdot T(w,w,\cdot).
\end{align*}
The individual estimates are
\begin{equation}\label{ineq:pre-slices}
 \|T_w\|_{\mathrm F}^2\le nc^2,\quad
 |\operatorname{tr}T_w|\le nc,\quad
 \|T(w,w,\cdot)\|\le c.
\end{equation}
Thus the last display is at most
\begin{equation}\label{ineq:pre-69}
 c^2(33n^2+72n)\le69c^2n^2\qquad(n\ge2),
\end{equation}
as required.
\end{proof}

\section{Proof of the main theorem}\label{sec:proof}
We first locate the saddle of the inverse Laplace integral and normalize the contour by the Hessian of $-\log p$. We then estimate the value and the first two moments of the normalized integral. Throughout this section $p$ is complete and $n\ge2$.

\subsection{Real roots, the Hessian, and the saddle}
For $e\in C$ and a real vector $v$, hyperbolicity gives real numbers $\lambda_j(v)$, counted with multiplicity, such that
\begin{equation}\label{eq:a-root-factorization}
 \frac{p(e+tv)}{p(e)}=\prod_{j=1}^m(1+t\lambda_j(v)).
\end{equation}
Zero values of $\lambda_j(v)$ account for possible loss of degree in this one-variable restriction. Differentiating at zero, we obtain
\begin{equation}\label{eq:a-root-derivatives}
 D\log p(e)[v]=\sum_j\lambda_j(v),\qquad
 H_p(e)[v,v]=\sum_j\lambda_j(v)^2.
\end{equation}
In particular the Hessian in \eqref{eq:a-hessian} is positive semidefinite. If its value on $v$ is zero, every root parameter in \eqref{eq:a-root-factorization} vanishes. The polynomial is then constant and nonzero on $e+\mathbb Rv$, so this whole line belongs to the same component $C$ of the nonvanishing set. Dividing the two ends of the line by arbitrarily large positive scalars puts both $v$ and $-v$ in $\overline C$. Pointedness therefore gives $v=0$. We have proved
\begin{equation}\label{ineq:a-hessian-positive}
 H_p(e)\succ0\qquad(e\in C).
\end{equation}

Fix $y\in\operatorname{int}K$ and consider
\begin{equation}\label{eq:a-objective}
 \Psi_y(x)=\langle x,y\rangle-\alpha\log p(x),\qquad x\in C.
\end{equation}
Its Hessian is $\alpha H_p(x)$, hence it is strictly convex. At any finite boundary point of $C$, the polynomial tends to zero, so $\Psi_y$ tends to positive infinity. Moreover there are constants $c_y>0$ and $M_p>0$ with
\begin{equation}\label{ineq:a-coercivity}
 \langle x,y\rangle\geq c_y\|x\|,\qquad
 p(x)\leq M_p\|x\|^m\qquad(x\in C).
\end{equation}
The first assertion follows by minimizing the pairing on $\overline C$ intersected with the unit sphere; the second follows by maximizing the absolute value of $p$ on the unit sphere and using homogeneity. Thus the linear part of \eqref{eq:a-objective} dominates its logarithmic part at infinity. There is a unique minimizer in $C$, and its critical-point equation is \eqref{eq:a-saddle}. 

\subsection{Normalization of the entire integration contour}
Fix this saddle and abbreviate $H_p(e)$ to $H_p$. Make the change of variables
\begin{equation}\label{eq:a-normalization}
 \xi=\alpha^{-1/2}H_p^{-1/2}z,
 \qquad b_j(z)=\lambda_j(H_p^{-1/2}z),\qquad r=\|z\|.
\end{equation}
Equation~\eqref{eq:a-root-derivatives} becomes
\begin{equation}\label{eq:a-root-norm}
 \sum_jb_j(z)^2=r^2.
\end{equation}
The normalized integral is
\begin{equation}\label{eq:a-normalized-integral}
 s:=\frac{q_\alpha(y)}{Q_\alpha(y)}
 =\frac1{(2\pi)^{n/2}}\int_{\mathbb R^n}W(z)\,dz,
\end{equation}
where
\begin{equation}\label{eq:a-W}
 W(z)=\exp\!\left(i\sqrt\alpha\sum_jb_j(z)\right)
       \prod_j(1+ib_j(z)/\sqrt\alpha)^{-\alpha}.
\end{equation}
The logarithms of the factors in \eqref{eq:a-W} are determined by continuation from $z=0$, where all factors equal one; their sum is therefore the tube logarithm in \eqref{eq:a-inverse}.

Write $W=M e^{i\phi}$. Its modulus and its centered phase are
\begin{align}
 M(z)&=\prod_j(1+b_j(z)^2/\alpha)^{-\alpha/2},
 \label{eq:a-modulus}\\
 \phi(z)&=\alpha\sum_j\left(
 \frac{b_j(z)}{\sqrt\alpha}
 -\arctan\frac{b_j(z)}{\sqrt\alpha}\right).
 \label{eq:a-phase}
\end{align}
Reality of $p$ and the chosen logarithm give $W(-z)=\overline{W(z)}$, so $s$ is real. Put
\begin{equation}\label{eq:a-envelope}
 K_\alpha(z)=(1+r^2/\alpha)^{-\alpha/2}.
\end{equation}
The inequalities $\log(1+u)\leq u$ for nonnegative $u$, and $\prod_j(1+u_j)\geq1+\sum_ju_j$ for nonnegative $u_j$, together with \eqref{eq:a-root-norm}, give the global bounds
\begin{equation}\label{ineq:a-modulus}
 e^{-r^2/2}\leq M(z)\leq K_\alpha(z).
\end{equation}

The leading term of the phase is determined by the cubic polynomial
\begin{equation}\label{eq:a-cubic}
 P(z)=\sum_jb_j(z)^3
 =\frac12D^3\log p(e)
 [H_p^{-1/2}z,H_p^{-1/2}z,H_p^{-1/2}z].
\end{equation}
The expression in terms of $D^3\log p$ shows that $P$ is a polynomial. By \eqref{eq:a-root-norm},
\begin{equation}\label{ineq:a-cubic-diagonal}
 |P(z)|\leq\sum_j|b_j(z)|^3\leq r^3.
\end{equation}
For every real $u$, integration of the derivative of the remainder yields
\begin{equation}\label{ineq:a-arctan}
 \left|u-\arctan u-\frac{u^3}{3}\right|
 =\left|\int_0^u\frac{-t^4}{1+t^2}\,dt\right|
 \leq\frac{|u|^5}{5}.
\end{equation}
Consequently
\begin{equation}\label{ineq:a-phase-remainder}
 \phi(z)=\frac{P(z)}{3\sqrt\alpha}+R(z),
 \qquad |R(z)|\leq\frac{r^5}{5\alpha^{3/2}}.
\end{equation}
Here $\sum_j|b_j|^5\le r^5$. Squaring the two terms and using $2|ab|\leq a^2+b^2$, we obtain
\begin{equation}\label{ineq:a-cosine-pointwise}
 1-\cos\phi\leq\frac{\phi^2}{2}
 \leq\frac{P(z)^2}{9\alpha}+
       \frac{r^{10}}{25\alpha^3}.
\end{equation}
We estimate the integral of $P^2$ using Lemma~\ref{lem:cubic}.

\subsection{The Gaussian comparison}
The radial majorant has total mass
\begin{equation}\label{eq:a-J}
 J=\frac1{(2\pi)^{n/2}}\int_{\mathbb R^n}K_\alpha(z)\,dz
 =\left(\frac\alpha2\right)^{n/2}
  \frac{\Gamma((\alpha-n)/2)}{\Gamma(\alpha/2)}.
\end{equation}
For a homogeneous polynomial $L$ of degree $2k$, integration in polar coordinates yields
\begin{equation}\label{eq:a-radial-polynomial}
 \frac1{(2\pi)^{n/2}}\int_{\mathbb R^n}L(z)K_\alpha(z)\,dz
 =R_k\,\mathbb E L(Z_0),\qquad
 R_k=J\prod_{j=1}^k\frac{\alpha}{\alpha-n-2j},
\end{equation}
provided $\alpha>n+2k$, where $Z_0$ is a standard Gaussian vector. Indeed the angular integrals for the two radial densities coincide, and the quotient of their radial integrals is evaluated by the beta integral after the substitution $u=r^2/\alpha$. In particular
\begin{equation}\label{eq:a-radial-moment}
 \mathcal M_{2k}:=\frac1{(2\pi)^{n/2}}\int r^{2k}K_\alpha(z)\,dz
 =R_k\prod_{j=0}^{k-1}(n+2j).
\end{equation}

The following estimates already hold in the larger range
\begin{equation}\label{ineq:a-moment-range}
 n\geq2,\qquad\alpha\geq64n^2,\qquad0\leq k\leq6.
\end{equation}
To bound $J$, let $U$ have the gamma distribution with shape $(\alpha-n)/2$ and rate one. Jensen's inequality applied to the convex function $u\mapsto u^{n/2}$ gives the first upper bound below. The logarithmic bound then follows from $-\log(1-t)\leq t/(1-t)$:
\begin{equation}\label{ineq:a-J-bounds}
 1\leq J\leq\left(\frac{\alpha}{\alpha-n}\right)^{n/2},
 \qquad
 \log J\leq\frac{n^2}{2(\alpha-n)}\leq\frac{n^2}{\alpha}.
\end{equation}
The lower bound is also immediate by integrating \eqref{ineq:a-modulus}. For the remaining factors of $R_k$, the quantities $(n+2j)/\alpha$ lie in $[0,1/2]$. On that interval, $-\log(1-t)\leq2t$. Hence
\begin{equation}\label{ineq:a-R-bounds}
 \begin{split}
 \log R_k
 &\leq\frac{n^2+2kn+2k(k+1)}{\alpha}
 \leq\frac{n^2+12n+84}{\alpha}\\
 &\leq\frac{28n^2}{\alpha}\leq\frac7{16}.
 \end{split}
\end{equation}
The third step uses $n\ge2$.
Comparing the power series for the exponential with the geometric series, we obtain
\begin{equation}\label{ineq:a-R-two}
 R_k\leq e^{7/16}\leq\frac{1}{1-7/16}
 =\frac{16}{9}<2.
\end{equation}

Apply Lemma~\ref{lem:cubic} to \eqref{eq:a-cubic} and \eqref{ineq:a-cubic-diagonal}. Its polarization constant is $c=9/2$. The two Gaussian bounds, followed by \eqref{eq:a-radial-polynomial} and \eqref{ineq:a-R-two}, give, for each unit vector $w$,
\begin{equation}\label{ineq:a-integrated-cubic}
 \begin{split}
 \frac1{(2\pi)^{n/2}}\int P(z)^2K_\alpha(z)\,dz
 &\leq\frac{1215}{2}n^2,\\
 \frac1{(2\pi)^{n/2}}\int\langle w,z\rangle^2P(z)^2K_\alpha(z)\,dz
 &\leq\frac{5589}{2}n^2.
 \end{split}
\end{equation}
For the radial remainder moments, $n+2j\leq(j+1)n$ follows from $n\geq2$. Using \eqref{eq:a-radial-moment}, we obtain
\begin{equation}\label{ineq:a-radial-high-moments}
 \mathcal M_{10}\leq2\cdot5!\,n^5=240n^5,
 \qquad
 \mathcal M_{12}\leq2\cdot6!\,n^6=1440n^6.
\end{equation}

Set $\delta=n^2/\alpha$. In the range \eqref{ineq:a-moment-range}, we have the useful remainder reduction
\begin{equation}\label{ineq:a-remainder-scale}
 \frac{n^3}{\alpha^2}\leq\frac{1}{4096n}\leq\frac1{8192}.
\end{equation}
By \eqref{ineq:a-modulus}, \eqref{ineq:a-cosine-pointwise},
\eqref{ineq:a-integrated-cubic}, and \eqref{ineq:a-radial-high-moments}, the integrated loss from the cosine is bounded by
\begin{equation}\label{ineq:a-cosine-loss}
 \begin{split}
 L_0&:=\frac1{(2\pi)^{n/2}}\int M(z)(1-\cos\phi(z))\,dz\\
 &\leq\frac{135}{2}\frac{n^2}{\alpha}
       +\frac{48}{5}\frac{n^5}{\alpha^3}
 <69\frac{n^2}{\alpha}.
 \end{split}
\end{equation}
The last inequality follows from \eqref{ineq:a-remainder-scale}.
Also \eqref{ineq:a-J-bounds} and the geometric-series estimate used above give
\begin{equation}\label{ineq:a-J-minus-one}
 J-1\leq e^\delta-1\leq\frac{\delta}{1-\delta}
 \leq\frac{64}{63}\delta<2\delta.
\end{equation}
The integral of $M$ lies between $1$ and $J$. Subtracting the nonnegative quantity $L_0$ therefore proves
\begin{equation}\label{ineq:a-s-error}
 1-L_0\leq s\leq J,
 \qquad |s-1|\leq69\frac{n^2}{\alpha}.
\end{equation}
We use this estimate together with the corresponding first- and second-moment bounds below.

\subsection{Two derivatives on a fixed contour}
The derivative estimates require the same moments, with a quadratic factor inserted. Define a real vector $c_0$ and a real symmetric matrix $\mathcal B$ by
\begin{equation}\label{eq:a-derivative-moments}
 c_0=\frac1{(2\pi)^{n/2}}\int z\,\operatorname{Im}W(z)\,dz,
 \qquad
 \mathcal B=\frac1{(2\pi)^{n/2}}\int zz^{\mathsf T}
               \operatorname{Re}W(z)\,dz.
\end{equation}
For a unit vector $w$, insert $\langle w,z\rangle^2$ into the estimate of $L_0$. The Gaussian term then has integral one. By radial symmetry,
\begin{equation}\label{eq:a-radial-direction}
 \frac1{(2\pi)^{n/2}}\int\langle w,z\rangle^2r^{10}K_\alpha(z)\,dz
 =\frac{\mathcal M_{12}}{n}.
\end{equation}
It follows that the loss in the lower quadratic-form bound is
\begin{equation}\label{ineq:a-B-lower-error}
 1-w^{\mathsf T}\mathcal B w
 \leq\frac{621}{2}\frac{n^2}{\alpha}
       +\frac{288}{5}\frac{n^5}{\alpha^3}
 <312\frac{n^2}{\alpha}.
\end{equation}
The last inequality again uses \eqref{ineq:a-remainder-scale}.

For the upper bound, $\operatorname{Re}W\leq M\leq K_\alpha$ implies
\begin{equation}\label{ineq:a-B-upper}
 \mathcal B\preceq B_0I,
 \qquad B_0=J\frac{\alpha}{\alpha-n-2}.
\end{equation}
Since $n+2\leq n^2$, the estimates already used for $R_k$ imply
\begin{equation}\label{ineq:a-B-zero}
 \log B_0\leq3\delta,\qquad
 B_0-1\leq\frac{3\delta}{1-3\delta}
 \leq\frac{192}{61}\delta<6\delta,
 \qquad B_0\leq\frac{64}{61}<2.
\end{equation}
Combining the lower and upper quadratic-form bounds, we obtain
\begin{equation}\label{ineq:a-B-error}
 \|\mathcal B-I\|\leq312\delta.
\end{equation}

For the first derivative, weighted Cauchy--Schwarz and
$\sin^2\phi\leq2(1-\cos\phi)$ show, for every unit vector $w$, that
\begin{equation}\label{ineq:a-cauchy}
 \begin{split}
 |\langle w,c_0\rangle|^2
 &\leq
 \left(\frac1{(2\pi)^{n/2}}\int\langle w,z\rangle^2M(z)\,dz\right)
 \left(\frac1{(2\pi)^{n/2}}\int M(z)\sin^2\phi(z)\,dz\right)\\
 &\leq2B_0L_0.
 \end{split}
\end{equation}
Taking the supremum over unit $w$ and using \eqref{ineq:a-cosine-loss} and \eqref{ineq:a-B-zero} yields
\begin{equation}\label{ineq:a-c-error}
 \|c_0\|^2\leq276\delta.
\end{equation}
Thus all three errors are controlled by the single constant in the theorem:
\begin{equation}\label{ineq:a-three-errors}
 |s-1|\leq d,\qquad
 \|\mathcal B-I\|\leq d,\qquad
 \|c_0\|^2\leq d,\qquad d=512\delta\leq\frac18.
\end{equation}
The hypothesis $\alpha\ge4096n^2=8\cdot512n^2$ gives $d\le1/8$, and hence $s>0$.

For the curvature calculation, choose the contour through the saddle $e$ at the point $y$ under consideration and hold it fixed during differentiation. By the symmetry $W(-z)=\overline{W(z)}$,
\begin{equation}\label{eq:a-fixed-contour}
 -D^2\log q_\alpha(y)
 =\frac1\alpha H_p^{-1/2}
 \left(\frac{\mathcal B}{s}+
       \frac{c_0c_0^{\mathsf T}}{s^2}\right)H_p^{-1/2}.
\end{equation}
Indeed $\nabla\log q_\alpha=e-\alpha^{-1/2}H_p^{-1/2}c_0/s$ on this contour, and subtraction of its square from $q_\alpha^{-1}\nabla^2q_\alpha$ yields \eqref{eq:a-fixed-contour}.

For $0\leq d\leq1/8$, the scalar comparisons needed in
\eqref{eq:a-fixed-contour} are
\begin{equation}\label{ineq:a-curvature-scalars}
 \frac{1-d}{1+d}\geq1-2d,
 \qquad
 \frac{1+d}{1-d}+\frac{d}{(1-d)^2}\leq1+4d.
\end{equation}
Both follow by clearing denominators. The bounds \eqref{ineq:a-three-errors} therefore imply \eqref{ineq:a-curvature}. Its lower coefficient is positive, giving strict log-concavity.

With the contour fixed, the integrand depends on $y$ only through $e^{\langle x+i\xi,y\rangle}$, and the envelope $K_\alpha$ has the moments needed to differentiate twice under the integral. The kernel is thus continuous on $\mathbb R^n$, supported in $K$, and positive on $\operatorname{int}K$. Since $K$ is the closure of its interior, continuity makes it nonnegative on the boundary as well. The Laplace identity is consequently a positive representation:
\begin{equation}\label{eq:a-laplace}
 p(x)^{-\alpha}=\int_K e^{-\langle x,y\rangle}q_\alpha(y)\,dy,
 \qquad x\in C.
\end{equation}
Differentiating this positive Laplace integral in directions $u_1,\ldots,u_k\in C$, we obtain
\begin{equation}\label{ineq:a-complete-monotonicity}
 (-1)^kD_{u_1}\cdots D_{u_k}p(x)^{-\alpha}
 =\int_K\prod_{j=1}^k\langle u_j,y\rangle
          e^{-\langle x,y\rangle}q_\alpha(y)\,dy\geq0.
\end{equation}
Finiteness and differentiation follow by choosing a sufficiently small neighborhood of $x$ inside $C$; polynomial moments are absorbed by the neighboring exponential factors. This completes the proof of Theorem~\ref{thm:main}.

\section{Consequences of positivity}\label{sec:consequences}
\subsection{Lineality and the one-dimensional case}\label{subsec:lineality}
Let now $C$ be a hyperbolicity cone whose closure need not be pointed, and let
\begin{equation}\label{eq:a-lineality}
 L=\overline C\cap(-\overline C).
\end{equation}
For an open convex cone, addition of its lineality space preserves the cone. Thus $x+\mathbb Rv$ lies in $C$ whenever $x\in C$ and $v\in L$. Its polynomial restriction is real-rooted and positive on the whole real line, so it is constant. Consequently $D_vp=0$ on the open set $C$ and hence as a polynomial identity. The polynomial descends to the quotient by $L$. The quotient hyperbolicity cone has pointed closure, and its dimension is the effective dimension of the problem.

Applying Theorem~\ref{thm:main} in the quotient gives a positive representing density there. Its pushforward to $L^\perp$ is the representing measure in the original space, and the strict log-concavity and curvature bounds hold on the relative interior of its support. Pullback to $C$ preserves complete monotonicity. 

In effective dimension one, after choosing the positive coordinate on the cone, the polynomial and kernel are
\begin{equation}\label{eq:a-one-dimensional}
 p(x)=cx^m,
 \qquad
 q_\alpha(y)=c^{-\alpha}\frac{y^{m\alpha-1}}{\Gamma(m\alpha)}
 \quad(y>0),\qquad c>0.
\end{equation}
The Laplace identity is the gamma integral. This gives complete monotonicity for every positive $\alpha$. Its logarithmic curvature is exactly
\begin{equation}\label{eq:a-one-dimensional-curvature}
 -\frac{d^2}{dy^2}\log q_\alpha(y)
 =\frac{m\alpha-1}{y^2}.
\end{equation}
Thus strict log-concavity in one dimension holds exactly when the exponent $m\alpha-1$ is positive.

\subsection{Exponential families}
The positive Riesz kernel supplies the base measure for the hyperbolic exponential family of \cite[Definition~3.6]{MSUZ}.

\begin{corollary}\label{cor:exponential-family}
Under the assumptions of Theorem~\ref{thm:main}, for every $x\in C$ the measure
\begin{equation}\label{eq:exponential-family-density}
 d\mu_{\alpha,x}(y)=p(x)^\alpha e^{-\langle x,y\rangle}q_\alpha(y)\,dy
\end{equation}
is a probability measure on $K$. For fixed $\alpha$, these measures form an exponential family with Laplace parameter $x$, base measure $q_\alpha(y)\,dy$, and log-partition function $A_\alpha(x)=-\alpha\log p(x)$. If $Y$ has distribution $\mu_{\alpha,x}$, then
\begin{equation}\label{eq:exponential-family-moments}
 \E_{\alpha,x}Y=\alpha\nabla\log p(x),\qquad
 \Cov_{\alpha,x}(Y)=\alpha H_p(x).
\end{equation}
Its density is strictly log-concave on $\operatorname{int}K$.
\end{corollary}
\begin{proof}
Normalization follows from \eqref{eq:a-laplace}. For $t$ in a neighborhood of zero with $x-t\in C$, the same identity yields
\begin{equation}\label{eq:exponential-family-cumulant}
 \log\E_{\alpha,x}e^{\langle t,Y\rangle}
 =\alpha\log p(x)-\alpha\log p(x-t).
\end{equation}
Differentiating at zero gives the mean and covariance in \eqref{eq:exponential-family-moments}. Multiplication by the positive exponential of an affine function leaves the logarithmic Hessian of $q_\alpha$ unchanged, so strict log-concavity follows from Theorem~\ref{thm:main}.
\end{proof}

\subsection{Concave roots and completely monotone powers}
\begin{corollary}[Concave root]\label{cor:concave-root}
Under the assumptions of Theorem~\ref{thm:main}, the function
$q_\alpha^{1/(m\alpha-n)}$ is concave on $\operatorname{int}K$.
\end{corollary}
\begin{proof}
Changing variables in \eqref{eq:a-inverse} and using contour independence, we obtain
\begin{equation}\label{eq:a-kernel-homogeneity}
 q_\alpha(ty)=t^{m\alpha-n}q_\alpha(y)\qquad(t>0).
\end{equation}
Put $k=m\alpha-n$ and $h=q_\alpha^{1/k}$. The parameter bound ensures $k>0$. The function $h$ is positive, homogeneous of degree one, and log-concave. For $x,y\in\operatorname{int}K$ and $0<t<1$, put
$A=(1-t)h(x)+th(y)$ and $\theta=th(y)/A$. Then
\begin{equation}\label{eq:a-root-normalization}
 \frac{(1-t)x+ty}{A}
 =(1-\theta)\frac{x}{h(x)}+\theta\frac{y}{h(y)}.
\end{equation}
Both normalized vectors on the right have $h$-value one. Log-concavity and homogeneity therefore give
\begin{equation}\label{ineq:a-root-concavity}
 h((1-t)x+ty)\geq A=(1-t)h(x)+th(y).
\end{equation}
The endpoints are immediate.
\end{proof}

\begin{corollary}\label{cor:hyperbolicity-cm}
Let $p$ be a nonconstant homogeneous real polynomial positive on a nonempty open convex cone $C$. Then $p$ is hyperbolic with respect to every point of $C$ if and only if $p^{-\alpha}$ is completely monotone on $C$ for some positive $\alpha$.
\end{corollary}
\begin{proof}
If $p$ is hyperbolic on $C$, enlarge $C$ to its containing hyperbolicity cone, reduce by lineality, and apply Theorem~\ref{thm:main} or the one-dimensional gamma integral. Complete monotonicity remains true on restriction to $C$.

Conversely, if a negative power is completely monotone, then $p$ has no zero in $T_C$ by \cite[Corollary~2.3]{SS}. Homogeneity turns this tube nonvanishing into real-rootedness. Indeed, if $p(v+te)=0$ for $e\in C$ and $t=a+ib$ with real $b\ne0$, multiplication of the argument by $(ib)^{-1}$ would give a zero at
\begin{equation}\label{eq:a-tube-zero}
 e-i(v+ae)/b\in T_C,
\end{equation}
a contradiction. Thus all zeros of $t\mapsto p(v+te)$ are real.
\end{proof}

\begin{corollary}\label{cor:half-plane-cm}
Let $P$ be a homogeneous real polynomial, positive on the open orthant,
with the half-plane property. Then $P^{-\alpha}$ is completely monotone
on the orthant for some $\alpha>0$.
\end{corollary}
\begin{proof}
If $P$ is nonconstant, the half-plane property and homogeneity imply
that $P$ is hyperbolic with respect to every point of the open orthant.
Apply Corollary~\ref{cor:hyperbolicity-cm} in that case; a positive
constant has completely monotone negative powers directly.
\end{proof}

\subsection{The Lorentz polynomial}\label{subsec:lorentz}
The quadratic order in dimension of the Gaussian cubic moment is already attained by the Lorentz polynomial. For
\begin{equation}\label{eq:a-Lorentz}
 p(x)=x_0^2-\sum_{j=1}^{n-1}x_j^2,
 \qquad e=(1,0,\ldots,0),
\end{equation}
we have $H_p(e)=2I$. If $Z_0$ is a standard Gaussian vector, the cubic in \eqref{eq:a-cubic} is
\begin{equation}\label{eq:a-Lorentz-cubic}
 P(Z_0)=\frac{Z_{0,0}^3+
                  3Z_{0,0}\sum_{j=1}^{n-1}Z_{0,j}^2}{\sqrt2}.
\end{equation}
Independence of $Z_{0,0}$ and the remaining coordinates, together with the Gaussian moments of orders two, four, and six, gives
\begin{equation}\label{eq:a-Lorentz-second-moment}
 \mathbb EP(Z_0)^2=\frac{9n^2+18n-12}{2}.
\end{equation}
On the other hand Lemma~\ref{lem:cubic} gives, for every complete hyperbolic polynomial,
\begin{equation}\label{ineq:a-cubic-uniform}
 \mathbb EP(Z_0)^2\leq\frac{1215}{4}n^2.
\end{equation}
For the normalized probability measure proportional to $K_\alpha(z)\,dz$, the exact multiplier for this cubic square is
\begin{equation}\label{eq:a-cubic-student}
 \mathbb E_{\nu_{\alpha,n}}P^2
 =\frac{\alpha^3}{(\alpha-n-2)(\alpha-n-4)(\alpha-n-6)}
   \mathbb EP(Z_0)^2.
\end{equation}
Its range of validity is $\alpha>n+6$. If $\alpha\geq2(n+6)$, each of the three factors lies between one and two, so
\begin{equation}\label{ineq:a-cubic-multiplier}
 1\leq\frac{\alpha^3}{(\alpha-n-2)(\alpha-n-4)(\alpha-n-6)}\leq8.
\end{equation}

The Riesz kernel is explicit in this case, so the Gaussian comparison can also be computed directly. For the Lorentz polynomial at $y=2\alpha e$, the explicit cone gamma integral in \cite[Proposition~5.6]{SS}, followed by Stirling's formula, yields
\begin{equation}\label{eq:a-Lorentz-Stirling}
 \frac{q_\alpha(2\alpha e)}{Q_\alpha(2\alpha e)}
 =\frac{2\pi\alpha^{2\alpha-n/2}e^{-2\alpha}}
        {\Gamma(\alpha)\Gamma(\alpha+1-n/2)}
 =1-\frac{3n^2-6n+4}{24\alpha}+O_n(\alpha^{-2}).
\end{equation}
The coefficient follows by adding the two $\alpha^{-1}$ terms in the shifted logarithmic Stirling expansion: $B_2(0)/2+B_2(1-n/2)/2=(3n^2-6n+4)/24$, where $B_2(t)=t^2-t+1/6$. The asymptotic is for each fixed dimension as $\alpha$ tends to infinity. The coefficient has order $n^2$, giving the dimension dependence of the Gaussian comparison. Positivity thresholds are determined separately by the admissible sets below.

For comparison, the known exact positivity ranges are as follows. According to \cite[Theorem~1.3(a)]{SS}, real symmetric determinants satisfy
\begin{equation}\label{eq:a-determinant-positivity}
 G(\det_m)=\{j/2:j\in\mathbb Z_{\geq0}\}
             \cup[(m-1)/2,\infty).
\end{equation}
For the elementary symmetric quadratic, \cite[Corollary~1.10]{SS} yields
\begin{equation}\label{eq:a-quadratic-positivity}
 G(E_{2,n})=\{0\}\cup[(n-2)/2,\infty).
\end{equation}
For a product of linear forms positive on the cone, every negative power with positive exponent is completely monotone, by multiplying the one-dimensional gamma representations. 

\section{The specialized V\'amos quartic}\label{sec:vamos}
Label the elements of $V_8$ in pairs so that its five nonbases are
$\{1,2,5,6\}$, $\{1,2,7,8\}$, $\{3,4,5,6\}$,
$\{3,4,7,8\}$, and $\{5,6,7,8\}$.  With this labeling its basis
polynomial has the specialization
\begin{align}\label{eq:vamos-polynomial}
 p(a,b,c,d)&=h_{V_8}(a,a,b,b,c,c,d,d)\\
 &=a^2b^2+4(a+b+c+d)(abc+abd+acd+bcd).
\end{align}
The five deleted monomials remove all squared pair products except
$a^2b^2$ from the degree-four elementary symmetric polynomial in the
eight repeated variables.  

Wagner and Wei established the half-plane property of $h_{V_8}$ by their criterion and the V\'amos calculation \cite[Theorem~3]{WagnerWei}.
Consequently $p$ is stable and hyperbolic in every positive direction.
For $e_\epsilon=(1,1,\epsilon,\epsilon)$, the univariate polynomials
$p(x+te_\epsilon)$ have only real roots.  Their coefficients converge
as $\epsilon$ tends to zero, and $p(1,1,0,0)=1$, so their degree-four
limits are also real-rooted.  Hence $p$ is hyperbolic with respect to $e=(1,1,0,0)$.

The obstruction to definite determinantal representations of positive
integer powers is due to Br\"and\'en
\cite[Theorem~3.3]{Branden}; the passage to this specialization
is recorded by Kummer \cite[Remark~1]{Kummer}.
The closure of the hyperbolicity cone is spectrahedral, as shown by Kummer \cite[Section~3]{Kummer}.

At $e$, direct differentiation yields
\begin{equation}\label{eq:vamos-gradient}
 p(e)=1,\qquad \nabla\log p(e)=(2,2,8,8)^{\mathsf T},
\end{equation}
\begin{equation}\label{eq:vamos-hessian}
 H_p(e)=\begin{pmatrix}
 2&0&4&4\\0&2&4&4\\4&4&56&40\\4&4&40&56
 \end{pmatrix}.
\end{equation}
The leading principal minors are $2,4,160,4096$, all positive.
Thus $H_p(e)$ is positive definite.  A lineality direction of the
hyperbolicity cone would be a zero direction of this Hessian; hence
$p$ is complete on this cone.
For
\begin{equation}\label{ineq:c-vamosthreshold}
 \alpha\ge4096\cdot4^2=65536,
\end{equation}
Theorem~\ref{thm:main} yields a positive, strictly log-concave density whose exponential tilt
at $e$ is a probability law with mean
$\alpha(2,2,8,8)^{\mathsf T}$ and covariance $\alpha H_p(e)$.
At $y_\alpha=\alpha(2,2,8,8)^{\mathsf T}$ the saddle is $e$ and
\begin{equation}\label{eq:vamos-Q}
 Q_\alpha(y_\alpha)=\frac{e^{4\alpha}}{64(2\pi\alpha)^2}.
\end{equation}
In particular,
\begin{equation}\label{ineq:c-vamoscomparison}
 \left|\frac{q_\alpha(y_\alpha)}{Q_\alpha(y_\alpha)}-1\right|
 \le\frac{8192}{\alpha}.
\end{equation}
Here $n=4$ and $\sqrt{\det H_p(e)}=64$.

\section{Spectral geometry}\label{sec:spectral}
Throughout the rest of the paper $p$ is complete, $C$ is its full hyperbolicity cone, $K=C^*$, and $C^\vee=\operatorname{int}K$. We write $V=\R^n$, $m=\deg p$, and $H_x=H_p(x)$. Unless stated otherwise $n,m\ge2$. Pairings between $V$ and its dual use the fixed Euclidean coordinates; the notation $\langle u,v\rangle_e=H_e[u,v]$ and $\|v\|_e$ refers to the Hessian metric.

Order the characteristic roots $\lambda_1^e(v)\ge\cdots\ge\lambda_m^e(v)$ and put $S_{k,e}(v)=\sum_j\lambda_j^e(v)^k$, with $S_{0,e}=m$. By the root factorization \eqref{eq:a-root-factorization},
\begin{equation}\label{eq:n-2-4}
 D^k\log p(e)[v^k]=(-1)^{k-1}(k-1)!S_{k,e}(v).
\end{equation}
Newton identities show that every $S_{k,e}$ is a homogeneous polynomial of degree $k$. Logarithmic homogeneity also implies
\begin{equation}\label{eq:n-2-5}
 H_e[v,v]=S_{2,e}(v),\qquad H_e[e,v]=D\log p(e)[v],\qquad H_e[e,e]=m.
\end{equation}

\subsection{Two-dimensional determinantal representations}

\begin{lemma}[Slice representation]\label{thm:n-2-1}
For every \(e\in C\) and \(u,v\in V\), there exist real symmetric \(m\times m\) matrices \(A,B\) such that
$p(e+su+tv)/p(e)=\det(I+sA+tB)$.

Their spectra are respectively \(\lambda^e(u)\) and \(\lambda^e(v)\), including zeros.
\end{lemma}
\begin{proof}
Every restriction of the polynomial on the left to a line through the origin has only real zeros, by hyperbolicity at $e$. The two-variable representation theorem of Helton and Vinnikov \cite[Theorem~2.2 and Section~4]{HV} therefore gives a real symmetric pencil of size equal to its degree. If this degree is smaller than $m$, adjoining zero blocks gives matrices of size $m$; products and repeated factors are represented by direct sums. Restriction to the coordinate axes identifies the two spectra, including zero roots. 

Differentiating the logarithmic determinant, we obtain
\begin{align}
 \operatorname{tr}(AB)&=H_e[u,v],\label{eq:n-2-9}\\
  \operatorname{tr}(A^k)&=S_{k,e}(u),\label{eq:n-2-10}\\
 \operatorname{tr}(A^2B)&=\tfrac12D^3\log p(e)[u,u,v].\notag
\end{align} If \(v\in \overline C\), its characteristic roots are nonnegative, so \(B\succeq0\). \end{proof}

The representation is attached to the pair $(u,v)$. We will also use the derivative identity $\frac1k DS_{k,e}(u)[v]=\tr(A^{k-1}B)$. Indeed, $S_{k,e}(u+tv)=\tr(A+tB)^k$, and the $k$ terms in its derivative have the same trace by cyclicity.

The sharpened Cauchy--Schwarz inequality \eqref{eq:n-2-12} is due to Bauschke, G\"uler, Lewis and Sendov \cite[Proposition~4.4]{BGLS}; the slice representation gives a short proof.

\begin{lemma}[Spectral metric inequalities]\label{thm:n-2-2}
With the \(H_e\) inner product,
\begin{equation}\label{eq:n-2-12}
\langle u,v\rangle_e
\le
\sum_{j=1}^m\lambda_j^e(u)\lambda_j^e(v),
\end{equation}
and consequently
\begin{equation}\label{eq:n-2-13}
\|\lambda^e(u)-\lambda^e(v)\|_2
\le
\|u-v\|_e.
\end{equation}

Moreover,
\begin{equation}\label{eq:n-2-14}
H_e\overline C\subseteq K.
\end{equation}

For every \(x\in C\),
\begin{equation}\label{eq:n-2-15}
H_x\succeq
\frac{H_e}{\lambda_1^e(x)^2},
\qquad
H_x^{-1}\preceq
\lambda_1^e(x)^2H_e^{-1}.
\end{equation}
\end{lemma}
\begin{proof}
By Theobald's trace theorem \cite{Theobald1975}, real symmetric matrices satisfy $\tr(AB)\le\sum_j\lambda_j(A)\lambda_j(B)$, with equality exactly when one orthogonal basis diagonalizes both matrices in decreasing eigenvalue order. Applying this to the slice representation proves \eqref{eq:n-2-12}. Since $\|v\|_e^2=\|\lambda^e(v)\|_2^2$, expansion of the squared distances proves \eqref{eq:n-2-13}.

If \(u,v\in \overline C\), the representing matrices are positive semidefinite, so $H_e[u,v]=\operatorname{tr}(AB)\ge0$. This is \eqref{eq:n-2-14}.

For \eqref{eq:n-2-15}, homogenize a slice representation: $p(se+tx+ru)/p(e)=\det(sI+tA+rB)$. Here \(A\succ0\), with eigenvalues \(a_i=\lambda_i^e(x)\), and $H_e[u,u]=\operatorname{tr}B^2$. At \(x\),
\[
H_x[u,u]
=
\operatorname{tr}(A^{-1}BA^{-1}B)
=
\sum_{i,j}\frac{B_{ij}^2}{a_ia_j}
\ge
\frac{\operatorname{tr}B^2}{\lambda_1^e(x)^2}.
\]
Taking inverses gives the second inequality. \end{proof}
\subsection{Analytic root branches}\label{sec:root-branches}
Write $p=c\prod_\nu h_\nu^{a_\nu}$ with distinct real irreducible homogeneous factors. Since $p(e)\ne0$, each $h_\nu(e)\ne0$. The change of variables $(t,v)\mapsto(t,te-v)$ is a polynomial automorphism, so $h_\nu(te-v)$ is irreducible in $\R[t,v]$. Its leading coefficient as a polynomial in $t$ is the nonzero constant $h_\nu(e)$; Gauss's lemma therefore gives irreducibility over $\R(v)[t]$. In characteristic zero this polynomial is separable, and its discriminant is nonzero. Distinct factors remain nonassociate over $\R(v)$, so their pairwise resultants are also nonzero.

Outside the zeros of these finitely many polynomials in $v$, all factor roots are simple and roots of different factors are disjoint. The implicit function theorem gives real analytic local branches, whose multiplicities in $p$ are the fixed integers $a_\nu$. This is a dense open subset of $V$. For $h\in\overline C$, a slice for $(v,h)$ realizes the roots of $v+th$ as the eigenvalues of $A+tB$ with $B\succeq0$. The min--max formula makes every ordered eigenvalue nondecreasing in $t$; the same monotonicity holds for each analytic branch where it is defined.

\section{Spectral probability and admissible exponents}\label{sec:admissible}
We study the set $G(p)$ defined in \eqref{eq:n-3-1}. For an arbitrary positive $\alpha\in G(p)$, let $\mu_\alpha$ be the unique positive measure with Laplace transform $p^{-\alpha}$. The measure may be singular; its exponential tilt still has a radial and angular decomposition.

\subsection{Exponential tilt and radial normalization}

Let \(\alpha\in G(p)\setminus\{0\}\), and let \(\mu_\alpha\) be its positive Laplace measure. At \(e\in C\), define
\begin{equation}\label{eq:n-3-2}
d\mathbb P_{\alpha,e}(y)
=
p(e)^\alpha e^{-\langle e,y\rangle}\,d\mu_\alpha(y).
\end{equation}

The calculation in Corollary~\ref{cor:exponential-family} gives the logarithmic moment generating function and the first two cumulants; it applies equally to this possibly singular Laplace measure.

Let $S_e=\{u\in K:u(e)=1\}$. This is compact.

Homogeneity and uniqueness of the Laplace transform imply $\mu_\alpha(tA)=t^{m\alpha}\mu_\alpha(A)$. The measure has no atom at the origin. Writing \(y=ru\), with \(r=y(e)>0\), its homogeneous polar decomposition is a constant multiple of $r^{m\alpha-1}\,dr\,d\sigma(u)$. After the tilt \eqref{eq:n-3-2},
\begin{equation}\label{eq:n-3-6}
Y=RU,
\qquad
R\sim\Gamma(m\alpha,1),
\qquad
R\perp U,
\qquad
U\in S_e.
\end{equation}

For a direct polar decomposition, put $\kappa=m\alpha$ and define a finite measure $\nu$ on $S_e$ by $\nu(E)=\mu_\alpha\{ru:0<r\le1,\ u\in E\}$. Finiteness follows from the exponential tilt at $e$. Homogeneity gives mass $t^\kappa\nu(E)$ to the same set with $r\le t$. Taking differences in $t$ identifies the product measure $\kappa r^{\kappa-1}\,dr\,d\nu(u)$. After multiplication by $p(e)^\alpha e^{-r}$, normalization gives the gamma law and its independence from the angular variable in \eqref{eq:n-3-6}.

\Needspace{8\baselineskip}
\subsection{Common spectral Dirichlet law}

\begin{theorem}\label{thm:n-3-1}
For \(\alpha>0\), membership in \(G(p)\) is equivalent to the existence of a probability measure on \(S_e\) such that, for every \(v\in V\),
\begin{equation}\label{eq:n-3-7}
U(v)\stackrel d=
\sum_{j=1}^mW_j\lambda_j^e(v),
\qquad
W\sim\operatorname{Dir}(\alpha,\ldots,\alpha).
\end{equation}

The probability measure is unique, and
\begin{equation}\label{eq:n-3-8}
\operatorname{conv}(\operatorname{supp}U)=S_e.
\end{equation}
\end{theorem}
\begin{proof}
Integration over the radial Gamma variable in \eqref{eq:n-3-6} yields
\begin{equation}\label{eq:n-3-9}
\left(\frac{p(x)}{p(e)}\right)^{-\alpha}
=
\mathbb E\,U(x)^{-m\alpha}.
\end{equation}

The Dirichlet integral identity is
\begin{equation}\label{eq:n-3-10}
\mathbb E\left(\sum_jW_ja_j\right)^{-m\alpha}
=
\prod_ja_j^{-\alpha},
\qquad a_j>0.
\end{equation}
Indeed, evaluate $\prod_j\int_0^\infty e^{-a_jt_j}t_j^{\alpha-1}\,dt_j$ first as a product of Gamma integrals, and then with \(t_j=rw_j\). The Jacobian is \(r^{m-1}\), and the radial integral has shape \(m\alpha\).

Set $a_j=1+t\lambda_j^e(v)$ for sufficiently small $t$, and compare \eqref{eq:n-3-9} at $x=e+tv$ with \eqref{eq:n-3-10}. Both random variables have compact support. Differentiating the uniformly convergent binomial series gives equality of every moment, since the coefficient $(m\alpha)_k/k!$ is nonzero. Polynomials are dense in the continuous functions on a compact interval containing both supports, so the two laws agree.

Conversely, apply \eqref{eq:n-3-7} with \(v=x-e\), and then \eqref{eq:n-3-10}, to obtain \eqref{eq:n-3-9}. Every \(u(x)^{-m\alpha}\) is a positive Gamma--Laplace transform, so its positive mixture is completely monotone.

The projections determine the characteristic function $\E e^{iU(v)}$ for every $v\in V$, and hence determine the probability measure uniquely.

Finally,
\begin{equation}\label{eq:n-3-11}
\max_{u\in S_e}u(v)
=
\inf\{s:se-v\in \overline C\}
=
\lambda_1^e(v).
\end{equation}
The Dirichlet distribution has the whole closed simplex as support, so its weighted average in \eqref{eq:n-3-7} has the same upper support endpoint. Equality of support functions proves \eqref{eq:n-3-8}. \end{proof}

The content of Theorem~\ref{thm:n-3-1} is that the directional laws, each given by Dirichlet weights on its own characteristic roots, are the projections of a single random linear functional.

For later use, if \(v\perp_{H_e}e\),
\begin{equation}\label{eq:n-3-12}
\mathbb EU(v)=0,
\qquad
\mathbb EU(v)^2
=
\frac{\|v\|_e^2}{m(m\alpha+1)}.
\end{equation}
This follows from the Dirichlet second moments and \(\sum_j\lambda_j^e(v)=0\).
\subsection{Dimension, degree, and exponent}

\begin{theorem}\label{thm:n-5-1}
If \(\alpha>0\) and \(p^{-\alpha}\in\mathrm{CM}(C)\), then
\begin{equation}\label{eq:n-5-7}
n\le m+\alpha m(m-1).
\end{equation}

Equality holds exactly when, up to a positive scalar and an invertible real linear change of variables, $p$ is a simple Euclidean Jordan determinant of rank $m$, and
\begin{equation}\label{eq:n-5-8}
\alpha=\frac d2,
\end{equation}
where \(d\) is its Peirce parameter.
\end{theorem}
\begin{proof}
Use the angular law of Theorem~\ref{thm:n-3-1}, and identify \(V^*\) with \(V\) through \(H_e\). Put $X=H_e^{-1}U$ and $Z=X-e/m$. Then \(Z\perp e\), and \eqref{eq:n-3-12} gives
\begin{equation}\label{eq:n-5-9}
\mathbb E\|Z\|_e^2
=
\frac{n-1}{m(m\alpha+1)}.
\end{equation}

For any \(u\in S_e\), set \(z=H_e^{-1}u-e/m\). Then $\|z\|_e^2=u(z)\le\lambda_1^e(z)$. Since \(\sum_j\lambda_j^e(z)=0\), Cauchy--Schwarz applied to the remaining roots gives $\lambda_1^e(z)\le\sqrt{(m-1)/m}\,\|z\|_e$. Hence
\begin{equation}\label{eq:n-5-10}
\|z\|_e^2\le\frac{m-1}{m}.
\end{equation}
Taking expectations proves \eqref{eq:n-5-7}.

The equality case is proved in Section~\ref{sec:equality}. \end{proof}
\subsection{A sharp degree-dependent gap at zero}

\begin{lemma}[A curved-graph incidence estimate]\label{thm:n-5-2}
Let \(I\) be a compact interval and \(\lambda\in C^2(I)\) satisfy $|\lambda''|\ge\kappa>0$. Then, for all real $a,b$ and every $\varepsilon>0$,
\begin{equation}\label{eq:n-5-12}
\left|
\{t\in I:|a+bt-\lambda(t)|<\varepsilon\}
\right|
\le
8\sqrt{\varepsilon/\kappa}.
\end{equation}
\end{lemma}
\begin{proof}
After changing signs, assume \(\lambda''\ge\kappa\). The function $g(t)=a+bt-\lambda(t)$ is strongly concave. The set \(-\varepsilon<g<\varepsilon\) has at most two interval components. If one component has length \(L\), the strong-concavity midpoint inequality gives $-\varepsilon+\kappa L^2/8<\varepsilon$, so \(L\le4\sqrt{\varepsilon/\kappa}\). Sum over the two components. \end{proof}

\begin{theorem}\label{thm:n-5-3}
If a nonlinear real irreducible factor occurs in \(p\) with multiplicity \(a\), then
\begin{equation}\label{eq:n-5-13}
\alpha\in G(p),\quad\alpha>0
\quad\Longrightarrow\quad
(m-a)\alpha\ge\frac12.
\end{equation}

Consequently, unless \(p\) is a product of real linear forms,
\begin{equation}\label{eq:n-5-14}
G(p)\cap
\left(0,\frac1{2(m-1)}\right)
=\varnothing.
\end{equation}
This bound is sharp for every degree \(m\ge2\).
\end{theorem}
\begin{proof}
By Section~\ref{sec:root-branches}, on a dense open set a root of the chosen irreducible factor is analytic and has multiplicity exactly $a$ in $p$.

Such a branch cannot be affine on an open set. Homogeneity would make it linear, say \(\lambda(v)=\ell(v)\), and translation by \(e\) would give \(\ell(e)=1\). The identity $h(v-\ell(v)e)=0$ would make $h$ vanish on an open subset of $\ker\ell$: the linear map $v\mapsto v-\ell(v)e$ is onto that hyperplane and maps open sets to relatively open sets. The restriction of $h$ to $\ker\ell$ would be zero as a polynomial, so $\ell$ would divide $h$, contrary to nonlinear irreducibility. Thus the Hessian of a root branch is nonzero somewhere. Choosing a direction in which its quadratic form is nonzero gives the asserted line segment after shrinking it.

Therefore some affine line \(v(t)=v_0+tw\) and compact interval \(I\) satisfy $|\lambda''(t)|\ge\kappa>0$. Shrink \(I\) so the other distinct roots remain separated and all root differences are uniformly bounded.

In the Dirichlet projection \eqref{eq:n-3-7}, combine the \(a\) weights at \(\lambda(t)\). The total remaining weight \(T\) has a Beta distribution with first parameter $(m-a)\alpha$. Thus $\Pr(T<\varepsilon/M) \ge c\varepsilon^{(m-a)\alpha}$ for sufficiently small $\varepsilon$, with constants uniform on $I$. In detail, if $b=(m-a)\alpha$ and $c_1=a\alpha$, the density of $T$ is $t^{b-1}(1-t)^{c_1-1}/\mathrm B(b,c_1)$. On $0<t<1/2$ the second factor has a positive lower bound, so integration from $0$ to $\varepsilon/M$ gives the stated estimate with $c>0$. Choose $M$ larger than every root difference on the compact interval. For each fixed $t$, the event in this Dirichlet realization implies that the weighted spectral average is within $\varepsilon$ of $\lambda(t)$. Equality of its law with the projection of $U$ therefore gives
\begin{equation}\label{eq:gap-probability}
 \Pr\bigl(|U(v_0)+tU(w)-\lambda(t)|<\varepsilon\bigr)
 \ge \Pr(T<\varepsilon/M)
 \ge c\varepsilon^{(m-a)\alpha}.
\end{equation}

Integrating over \(t\), then applying Fubini and Lemma~\ref{thm:n-5-2}, gives $c|I|\varepsilon^{(m-a)\alpha} \le 8\sqrt{\varepsilon/\kappa}$. Let \(\varepsilon\downarrow0\). This proves \eqref{eq:n-5-13}, and then \eqref{eq:n-5-14}.

For sharpness, take
\[
p_m(x,t,z)=x^{m-2}(xt-z^2),
\qquad
C=\{x>0,\ xt>z^2\}.
\]
Let $\eta=(m-1)\alpha-\frac12\ge0$. On \(u>0,w\in\mathbb R,r\ge0\), use the positive measure
\begin{equation}\label{eq:n-5-15}
\frac{u^{\alpha-3/2}}{\Gamma(\alpha)\sqrt{4\pi}}\,
du\,dw\,
\frac{r^{\eta-1}}{\Gamma(\eta)}\,dr,
\end{equation}
with the last factor interpreted as \(\delta_0(dr)\) when \(\eta=0\). Push it forward under $(u,w,r)\longmapsto(r+w^2/(4u),u,w)$. Its support lies in the closed dual cone.

In the pairing \(xs+tu+zw\), integration first in \(r\), then in \(w\), then in \(u\), gives
\[
x^{-\eta-1/2}(t-z^2/x)^{-\alpha}
=
x^{-(m-2)\alpha}(xt-z^2)^{-\alpha}.
\]
For $(x,t,z)\in C$, completion of the square gives
\[
 \int_{\R}e^{-xw^2/(4u)-zw}\,dw
 =\sqrt{\frac{4\pi u}{x}}\,e^{uz^2/x}.
\]
The remaining gamma integral has rate $t-z^2/x>0$, and the preceding expression follows, including $\eta=0$. The resulting Laplace integral is finite at each point of $C$, so the pushforward is locally finite: on a compact set the exponential weight is bounded below. Its support is dual since $s=r+w^2/(4u)$ gives $su\ge w^2/4$. The Laplace integrands are nonnegative, so Tonelli's theorem justifies each change in the order of integration.
Thus
\begin{equation}\label{eq:n-5-16}
G(p_m)=
\{0\}\cup
\left[\frac1{2(m-1)},\infty\right).
\end{equation}
This proves sharpness. \end{proof}

Writing $\Delta_1(x,t,z)=x$ and $\Delta_2(x,t,z)=xt-z^2$, the polynomial $p_m=\Delta_1^{m-2}\Delta_2$ is a generalized power on $\operatorname{Sym}_2$, and \eqref{eq:n-5-16} is also a special case of Gindikin's positivity theorem for generalized-power Riesz distributions \cite[Theorem~1]{Gindikin}.

In particular, zero is an accumulation point of positive admissible exponents exactly when \(p\) is a product of real linear forms. Such products have positive Gamma representations at every positive exponent.
\subsection{Accumulation at zero}\label{sec:gap}
Theorem~\ref{thm:n-5-3}, together with the lineality reduction of Section~\ref{subsec:lineality}, gives the following statement for arbitrary hyperbolic polynomials. The notation $G(p)$ remains that of \eqref{eq:n-3-1}.

\begin{corollary}[Accumulation at zero]\label{thm:gap}
Let $p$ be a homogeneous hyperbolic polynomial of positive degree $m$, positive on $C$. If $G(p)$ contains positive parameters tending to zero, then
\begin{equation}\label{eq:a-linear-factorization}
 p(x)=c\prod_{\ell=1}^r L_\ell(x)^{m_\ell},
\end{equation}
where $c>0$, the $L_\ell$ are real linear forms positive on $C$, and the $m_\ell$ are positive integers. Consequently, if $p$ has a nonlinear real irreducible factor, there is an $\epsilon(p)>0$ such that $G(p)\cap(0,\epsilon(p))=\varnothing$. Conversely, a polynomial of the form \eqref{eq:a-linear-factorization} has every nonnegative parameter in $G(p)$.
\end{corollary}

\begin{proof}
Pass to the lineality quotient as in Section~\ref{subsec:lineality}. The resulting complete polynomial has the same admissible exponents. If its quotient has dimension one, it is a power of a real linear form; otherwise Theorem~\ref{thm:n-5-3} excludes a nonlinear irreducible factor when positive admissible exponents tend to zero. Pulling back the linear factors gives \eqref{eq:a-linear-factorization}. The converse follows from their gamma-integral representations.
\end{proof}

\begin{remark}[Homogeneity is essential]\label{rem:a-homogeneity}
In the paragraph following Lemma~2.4, Scott--Sokal \cite[Lemma~2.4]{SS} give the nonhomogeneous quartic example $P(x)=((x+a)^2+b^2)(x+c)^2$, where $0<c\leq a$ and $b\ne0$. The polynomial has nonreal zeros, yet all its negative powers with positive exponent are completely monotone on the positive half-line. Indeed the inverse Laplace density of $(\log P)'$ is
\begin{equation}\label{ineq:a-quartic-density}
 2e^{-at}\cos(bt)+2e^{-ct}
 \geq2(e^{-ct}-e^{-at})\geq0\qquad(t\geq0).
\end{equation}
Lemma~2.4 of that source gives the stated complete monotonicity. The homogeneity hypothesis in Corollary~\ref{thm:gap} is therefore essential. 
\end{remark}

\part{Jordan rigidity and Legendre duality}
The logarithmic gradient supplies the dual coordinates used in this part. We first show that vanishing of the fourth-order defect at a single point recovers a Euclidean Jordan algebra and the polynomial itself. The resulting identity for the Hessian determinant then gives the Monge--Amp\`ere and polynomial-duality classifications. Finally, the saddle-point formula turns these classifications into rigidity statements for the Riesz kernels.

\section{The logarithmic gradient and multiplicative duality}\label{sec:legendre}

\begin{proposition}\label{thm:n-2-3}
The logarithmic gradient $\tau_p(x)=D\log p(x)$ is a real-analytic diffeomorphism from $C$ onto $C^\vee$.

Its multiplicative Legendre transform is positive, real-analytic, homogeneous of degree \(m\), and satisfies
\begin{align}
 \langle x,\tau_p(x)\rangle&=m, \qquad H_xx=\tau_p(x),\label{eq:n-2-16}\\
 D\log p_*(\tau_p(x))&=x,\label{eq:n-2-17}\\
 -D^2\log p_*(\tau_p(x))&=H_x^{-1}.\label{eq:n-2-18}
\end{align}
In particular, $(p_*)_*=p$. \end{proposition}
\begin{proof}
For nonzero $v\in\overline C$, the roots relative to $x\in C$ are nonnegative and not all zero, so $D\log p(x)[v]=\sum_j\lambda_j^x(v)>0$. Thus $\tau_p(x)\in C^\vee$. Conversely, for $y\in C^\vee$, the strictly convex function $f_y(x)=\langle x,y\rangle-\log p(x)$ tends to $+\infty$ at the finite boundary of $C$. There is $c_y>0$ such that $\langle x,y\rangle\ge c_y\|x\|$ on $\overline C$, while $p(x)\le M_p\|x\|^m$; hence $f_y$ also tends to $+\infty$ at infinity. Its unique minimizer satisfies $\tau_p(x)=y$. Since $D\tau_p=-H$ is invertible, the inverse function theorem gives a real-analytic inverse.

Euler's identity gives \eqref{eq:n-2-16}. Differentiating $\log p_*(\tau_p(x))=-\log p(x)$ and using $H_xx=\tau_p(x)$ proves \eqref{eq:n-2-17}; a second differentiation gives \eqref{eq:n-2-18}. Finally, $\tau_p(tx)=t^{-1}\tau_p(x)$ gives homogeneity, and the inverse-gradient identity gives involutivity. These are the multiplicative Legendre identities of Etingof, Kazhdan and Polishchuk \cite[Proposition~3.5]{EKP}.
\end{proof}
\section{One-point Jordan rigidity}\label{sec:rigidity}
A Euclidean Jordan algebra is a finite-dimensional real commutative algebra with unit, a positive definite invariant inner product, and the identity $u^2\circ(u\circ v)=u\circ(u^2\circ v)$. Its spectral decomposition writes every element as $\sum_{i=1}^r t_i c_i$ in a Jordan frame of primitive orthogonal idempotents; the determinant is $\Delta(x)=\prod_i t_i$. On a simple algebra the off-diagonal Peirce spaces have a common dimension $d$, and $\dim J=r+d r(r-1)/2$ \cite[Chapters~II--V]{FK}.

Fix $e\in C$ and use its Hessian inner product, suppressing the subscript when the base point is fixed. Define a bilinear multiplication by
\begin{equation}\label{eq:n-4-1}
\langle u\circ v,w\rangle
=
\frac12D^3\log p(e)[u,v,w].
\end{equation}

The symmetry of the third derivative makes this product commutative and the inner product invariant: $\langle u\circ v,w\rangle = \langle u,v\circ w\rangle$, and logarithmic homogeneity gives $e\circ v=v$. Write $v^2=v\circ v$ and $v^3=v\circ v^2$. Logarithmic differentiation gives
\begin{equation}\label{eq:n-4-4}
S_3(v)=\langle v^2,v\rangle,
\qquad
v^2=\frac13\nabla_HS_3(v).
\end{equation}

The quartic defect $\mathcal D_e(v)=S_{4,e}(v)-\|v^2\|_e^2$ can equivalently be written as
\begin{equation}\label{eq:n-4-6}
\mathcal D_e(v)
=
-\frac16D^4\log p(e)[v^4]
-
\left\|
\frac12H_e^{-1}D^3\log p(e)[v,v,\cdot]
\right\|_e^2.
\end{equation}

\begin{theorem}[One-point Jordan rigidity]\label{thm:n-4-1}
For every $e\in C$ and $v\in V$, $\mathcal D_e(v)\ge0$.

The following conditions are equivalent:

\begin{enumerate}
\item \(\mathcal D_e\equiv0\) at some \(e\in C\).
\item \(H_e\overline C=K\) at some \(e\in C\).
\item After an invertible real linear change of coordinates,
\begin{equation}\label{eq:n-4-8}
p=c\prod_{a=1}^s\Delta_{J_a}^{k_a},
\qquad
c>0,\quad k_a\in\mathbb Z_{>0},
\end{equation}
where \(V=\bigoplus_aJ_a\), the \(J_a\) are simple Euclidean Jordan algebras, and each determinant uses only its own variable group.
\item For \(F=-\log p\), the Levi--Civita covariant derivative $\widehat\nabla D^3F$ vanishes at some point of \(C\).

\end{enumerate}
If one condition holds, conditions 1, 2, and 4 hold at every point of \(C\).
\end{theorem}
\begin{proof}[Proof of Theorem~\ref{thm:n-4-1}]
The proof proceeds from the square and cube spectra to an eighth-order identity, which implies the Jordan identity. The Jordan spectral decomposition and Newton identities then recover the polynomial, including its integer multiplicities. We finish by relating this reconstruction to Hessian self-duality and the covariant derivative of the cubic tensor.

\subsection{Nonnegativity and the square spectrum}

The defect compares $S_4(u)$ with $\|u^2\|^2$. A slice in the directions $u$ and $u^2$ turns this into a comparison between the matrix representing $u^2$ and the square of the matrix representing $u$. Apply Lemma~\ref{thm:n-2-1} to the pair \((u,u^2)\). The representing matrices satisfy
\[
\operatorname{tr}B^2=\|u^2\|^2,
\qquad
\operatorname{tr}A^4=S_4(u),
\]
and $\operatorname{tr}(A^2B)=\langle u^2,u^2\rangle$. Therefore
\begin{equation}\label{eq:n-4-9}
\mathcal D_e(u)=\operatorname{tr}(B-A^2)^2\ge0.
\end{equation}

Assume henceforth that \(\mathcal D_e\equiv0\). Equation \eqref{eq:n-4-9} forces \(B=A^2\), giving
\begin{align}
 \frac{p(e+su+tu^2)}{p(e)} &= \det(I+sA+tA^2),\label{eq:n-4-10}\\
 \lambda(u^2) &= \bigl(\lambda_1(u)^2,\ldots,\lambda_m(u)^2\bigr)^\downarrow,\label{eq:n-4-11}\\
 S_k(u^2)&=S_{2k}(u)\qquad(k\ge1).\label{eq:n-4-12}
\end{align}

\subsection{The cube spectrum}

On a suitable slice, $u^3$ will likewise be represented by the cube of the matrix representing $u$. Differentiate $S_4(u)=\|u^2\|^2$. By invariance,
\begin{equation}\label{eq:n-4-13}
\frac14DS_4(u)[w]=\langle u^3,w\rangle.
\end{equation}

Set \(a=(u^2)^2\). From \eqref{eq:n-4-10}, $S_4(u+tu^2) = \operatorname{tr}(A+tA^2)^4$. On the other hand,
\[
S_4(u+tu^2)
=
\|u^2+2tu^3+t^2a\|^2.
\]

Equating the coefficients of \(t\) and \(t^2\) yields
\begin{align}
 \langle u^2,u^3\rangle&=S_5(u),\label{eq:n-4-14}\\
 6S_6(u) &= 4\|u^3\|^2+2\langle u^2,a\rangle.\label{eq:n-4-15}
\end{align}
But $\langle u^2,a\rangle = S_3(u^2)=S_6(u)$. Hence
\begin{equation}\label{eq:n-4-16}
\|u^3\|^2=S_6(u).
\end{equation}

Now apply Lemma~\ref{thm:n-2-1} anew to \((u,u^3)\), with matrices \(A,C\). Equation \eqref{eq:n-4-13} implies
\[
\operatorname{tr}(A^3C)
=
\frac14DS_4(u)[u^3]
=
S_6(u).
\]
Also, $\operatorname{tr}A^6=S_6(u)$ and $\operatorname{tr}C^2=S_6(u)$. Therefore $\operatorname{tr}(C-A^3)^2=0$, so \(C=A^3\). We have recovered
\begin{equation}\label{eq:n-4-17}
\frac{p(e+su+tu^3)}{p(e)}
=
\det(I+sA+tA^3).
\end{equation}

\subsection{Fourth-power associativity}
The square- and cube-spectrum identities allow us to compute both squared norms and the inner product of $(u^2)^2$ and $u\circ u^3$; these three quantities force the vectors to coincide. Take $u\in C$ and represent $(u^2,u^3)$ by matrices $U,V$. Their spectra are $\lambda_j(u)^2$ and $\lambda_j(u)^3$, and \eqref{eq:n-4-14} gives $\tr(UV)=\sum_j\lambda_j(u)^5$. The equality case of Theobald's theorem \cite{Theobald1975} gives a common ordered orthogonal eigenbasis. Since the roots are positive, the paired eigenvalues satisfy $\lambda_j(u)^3=(\lambda_j(u)^2)^{3/2}$, including repeated eigenvalues. Thus $V=U^{3/2}$ and
\begin{equation}\label{eq:n-4-19}
 \langle u^2,(u^3)^2\rangle=\tr(UV^2)=S_8(u).
\end{equation}

Put $a=(u^2)^2$ and $b=u\circ u^3$. The square-spectrum identity \eqref{eq:n-4-12} gives the norm of $a$. Comparing the coefficient of $t^2$ in $S_4(u+tu^3)=\|(u+tu^3)^2\|^2$, using \eqref{eq:n-4-17} and then \eqref{eq:n-4-19}, gives the norm of $b$:
\begin{align}
 \|a\|^2&=S_4(u^2)=S_8(u),\label{eq:n-4-20}\\
 6S_8(u)&=4\|b\|^2+2\langle u^2,(u^3)^2\rangle,\notag\\
 \|b\|^2&=S_8(u).\label{eq:n-4-21}
\end{align}
To compute their inner product, differentiate $S_3(u^2)=S_6(u)$ and use \eqref{eq:n-4-4}. Taking $w=u^3$ and applying \eqref{eq:n-4-17} yields
\begin{align}
 DS_6(u)[w]&=6\langle u\circ a,w\rangle,\label{eq:n-4-22}\\
 DS_6(u)[u^3]&=6\tr(A^5A^3)=6S_8(u),\notag\\
 \langle a,b\rangle&=S_8(u).\label{eq:n-4-23}
\end{align}
The three inner products in \eqref{eq:n-4-20}--\eqref{eq:n-4-23} therefore give
\begin{align}
 \|(u^2)^2-u\circ(u\circ u^2)\|^2&=0,\label{eq:n-4-24}\\
 (u^2)^2&=u\circ(u\circ u^2).\label{eq:n-4-25}
\end{align}
The last identity holds on the nonempty open set $C$, and hence on all of $V$ as a polynomial identity.

\subsection{The Jordan identity}

Write $L_uv=u\circ v$. By invariance each $L_u$ is self-adjoint, so one linearization of \eqref{eq:n-4-25} suffices to obtain the Jordan identity. Differentiating it, we obtain
\[
4L_{u^2}L_u
=
L_{u^3}+L_uL_{u^2}+2L_u^3.
\]
Subtract its adjoint:
\begin{equation}\label{eq:n-4-26}
5[L_{u^2},L_u]=0.
\end{equation}
This is the Jordan identity. The algebra is formally real as well, since
\[
\sum_i v_i^2=0
\quad\Longrightarrow\quad
0=\left\langle e,\sum_i v_i^2\right\rangle
=\sum_i\|v_i\|^2,
\]
so every $v_i=0$. Thus $(V,\circ,e)$ is a Euclidean Jordan algebra.

\subsection{Recovery of the polynomial}

The power sums of the characteristic roots determine $p$, and they can now be computed inside the Jordan algebra. Decompose $V=\bigoplus_aJ_a$ into simple ideals. On each $J_a$, the invariant inner product is $k_a>0$ times the standard trace inner product: represent the inner product relative to the trace form by a positive self-adjoint operator \(T\). Invariance makes \(T\) commute with every multiplication operator. Its eigenspaces are therefore ideals, and simplicity forces \(T\) to be scalar. Distinct simple ideals are orthogonal: for $x\in J_a$, $y\in J_b$ and $a\ne b$, invariance and the unit $e_a$ of $J_a$ give $\langle x,y\rangle=\langle e_a\circ x,y\rangle=\langle x,e_a\circ y\rangle=0$. The spectral and simple-ideal decompositions used here are those of Euclidean Jordan algebras; see \cite[Chapters~II--V]{FK}.

Let \(c\) be a primitive idempotent in \(J_a\). Since \(c^2=c\), the multiset \(\lambda(c)\) is invariant under componentwise squaring, by \eqref{eq:n-4-11}. Squaring permutes this finite multiset. Each of its entries $t$ therefore satisfies $t^{2^r}=t$ for some $r\ge1$, whose only real solutions are $0$ and $1$. Therefore
\begin{align}
 k_a&=\|c\|^2=S_2(c)\in\mathbb Z_{>0}.\label{eq:n-4-27}\\
 \sum_a k_a\operatorname{rank}J_a &= \|e\|^2=m.\label{eq:n-4-28}
\end{align}

Let \(\mu_{a,i}(v_a)\) denote the Jordan characteristic values in \(J_a\). Repeated use of the square-spectrum identity gives, for \(M=2^r\), \(r\ge1\),
\begin{equation}\label{eq:n-4-29}
S_M(v)
=
\|v^{M/2}\|^2
=
\sum_a k_a\sum_i\mu_{a,i}(v_a)^M.
\end{equation}

Choose one such \(M\ge m\), and apply \eqref{eq:n-4-29} to \(v+te\). Both characteristic spectra translate by \(t\). Comparing coefficients of \(t^{M-j}\) yields
\begin{equation}\label{eq:n-4-30}
S_j(v)
=
\sum_a k_a\sum_i\mu_{a,i}(v_a)^j
\qquad(0\le j\le m).
\end{equation}

These are power sums of multisets with the same total multiplicity $m$. Starting from $e_0=1$, Newton's recursion
\[
 j e_j=\sum_{i=1}^j(-1)^{i-1}e_{j-i}S_i\qquad(1\le j\le m)
\]
determines all their elementary symmetric functions. Hence
\begin{equation}\label{eq:n-4-31}
\frac{p(te-v)}{p(e)}
=
\prod_a\prod_i
\bigl(t-\mu_{a,i}(v_a)\bigr)^{k_a}.
\end{equation}
Taking constant terms,
\begin{equation}\label{eq:n-4-32}
p(v)=p(e)\prod_a\Delta_{J_a}(v_a)^{k_a}.
\end{equation}
This proves $1\Rightarrow3$. Conversely, suppose $p$ has the product form \eqref{eq:n-4-8}. In each simple factor a hyperbolicity direction has Jordan eigenvalues of one sign. Indeed, if two eigenvalues $t_i,t_j$ had opposite signs, choose $w\in J_{ij}$ with $w^2=c_i+c_j$. The restriction of the determinant in that direction contains the factor $t_it_j-s^2$, contradicting the real-zero property at the hyperbolicity point. Changing the signs of the factor coordinates as necessary therefore identifies the given full cone $C$ with the product of positive Jordan cones. At the product of their units the multiplication \eqref{eq:n-4-1} is the Jordan multiplication and $S_4(v)=\|v^2\|^2$. The quadratic representation $P(a)=2L_a^2-L_{a^2}$ of a Euclidean Jordan algebra satisfies $P(a)\Omega=\Omega$ for invertible $a$ and $\Delta(P(a)x)=\Delta(a)^2\Delta(x)$ \cite[Chapters~III--IV]{FK}. Taking $a=x^{-1/2}$ carries a positive point $x$ to the unit. The Hessian, cubic product, and fourth-order defect transform under this linear map, while the constant multiplier of the determinant disappears after logarithmic differentiation. The equality therefore holds at every point of $C$.

\subsection{One-point Hessian self-duality}

Self-duality forces the $H_e$-gradients of the distinct root branches at a generic point to be mutually orthogonal elements of $\overline C$; this orthogonality makes the defect vanish. Assume $H_e\overline C=K$. At a generic point $x$, let \(a_i(x)\) be the distinct characteristic roots, with multiplicities \(r_i\), and put $c_i=\nabla_{H_e}a_i(x)$. Section~\ref{sec:root-branches} supplies these branches on a dense open set, with fixed multiplicities even when $p$ has repeated factors.

By the monotonicity established in Section~\ref{sec:root-branches}, $H_ec_i$ is nonnegative on $\overline C$; the assumed self-duality therefore gives $c_i\in\overline C$.

Translation, homogeneity, and the first two power sums give
\begin{align}
 \langle e,c_i\rangle&=1, \qquad \langle x,c_i\rangle=a_i,\label{eq:n-4-33}\\
 \sum_i r_ic_i&=e, \qquad \sum_i r_ia_ic_i=x.\label{eq:n-4-34}
\end{align}

Let
\[
G_{ij}=\langle c_i,c_j\rangle\ge0,
\qquad
P_{ij}=r_jG_{ij}.
\]
Equations \eqref{eq:n-4-33}--\eqref{eq:n-4-34} give $P\mathbf1=\mathbf1$ and $Pa=a$. Symmetry of $G$ also gives $\sum_i r_iP_{ij}=r_j$. Expanding the square, we obtain
\[
 \sum_{i,j}r_iP_{ij}(a_i-a_j)^2
 =2\sum_i r_i a_i^2-2\sum_i r_i a_i(Pa)_i=0.
\]
Equivalently, $\sum_{i,j}r_ir_jG_{ij}(a_i-a_j)^2=0$. The \(a_i\) are distinct, so \(G_{ij}=0\) for \(i\ne j\), and the row sums give \(G_{ii}=1/r_i\).

Since
\[
x^2=\frac13\nabla_HS_3(x)=\sum_i r_ia_i^2c_i,
\]
we obtain $\|x^2\|^2=\sum_i r_ia_i^4=S_4(x)$. The polynomial $\mathcal D_e$ therefore vanishes identically. Conversely, a product of Jordan positive cones is self-dual for its weighted trace inner product. This proves conditions 2 and 3 are equivalent.

\subsection{The covariant fourth derivative}

Finally, we identify six times the defect with the diagonal of the covariant derivative of the cubic form, the tensor whose vanishing on open sets was studied by Hildebrand \cite{Hildebrand}. For \(F=-\log p\), let \(g=D^2F=H\). In affine coordinates the Levi--Civita symbols are $\Gamma^a_{ij}=\tfrac12g^{ab}F_{ijb}$. Subtracting the three connection terms from the affine derivative of $F_{ijk}$ gives
\begin{equation}\label{eq:n-4-36}
(\widehat\nabla D^3F)_{ijkl}
=
F_{ijkl}
-\frac12g^{ab}
\left(
F_{ija}F_{klb}
+
F_{ika}F_{jlb}
+
F_{ila}F_{jkb}
\right).
\end{equation}
This tensor is symmetric, and its diagonal value is
\begin{equation}\label{eq:n-4-37}
(\widehat\nabla D^3F)_e[v^4]
=
6\mathcal D_e(v).
\end{equation}
Polarization proves the remaining equivalence. For the equation on an open set, see Hildebrand \cite{Hildebrand}.
\end{proof}
\subsection{Equality in the dimension bound}\label{sec:equality}
\begin{proof}[Equality in Theorem~\ref{thm:n-5-1}]
Suppose equality holds in Theorem~\ref{thm:n-5-1}, and retain its random vectors $X=H_e^{-1}U$ and $Z=X-e/m$. The slack in \eqref{eq:n-5-10} is continuous and nonnegative, so it vanishes throughout the support. Equality in Cauchy--Schwarz forces $\lambda^e(Z)=((m-1)/m,-1/m,\ldots,-1/m)$. Thus
\begin{equation}\label{eq:n-5-11}
\lambda^e(X)=(1,0,\ldots,0),
\qquad X\in \overline C.
\end{equation}

Because the angular support generates \(S_e\), $H_e^{-1}K\subseteq \overline C$. Together with \eqref{eq:n-2-14}, this gives \(H_e\overline C=K\). Theorem~\ref{thm:n-4-1} yields
$p=c\prod_a\Delta_a^{k_a}$.

Every extreme point of the compact base $H_e^{-1}S_e$ belongs to the compact support generating it: by Carath\'eodory's theorem it is a convex combination of at most $n+1$ support points, and extremality forces every point with positive coefficient to equal that extreme point. Thus every extreme ray of $\overline C$ has $p$-rank one. A primitive idempotent in $J_a$ has its nonzero $p$-root repeated $k_a$ times, so $k_a=1$ for every $a$. Let $r_a=\operatorname{rank}J_a$. Restrict the positive Laplace representation to one factor while fixing the others at their units. Directions supported on one factor are limits of interior directions, so \(\Delta_a^{-\alpha}\) is completely monotone. For a factor of rank at least two apply the inequality in Theorem~\ref{thm:n-5-1}; a rank-one factor has $n_a=r_a=1$ and satisfies the same inequality directly. Summing gives $n\le m+\alpha\sum_a r_a(r_a-1)$. If there were at least two factors, $\sum_a r_a(r_a-1)<m(m-1)$, contradicting equality. There is therefore one simple factor, whose dimension $n=m+\frac d2m(m-1)$ gives $\alpha=d/2$. Conversely, the positive Riesz measure at $\alpha=d/2$ exists on every simple Euclidean Jordan cone \cite[Theorem~1 and the following example]{Gindikin}. The Peirce dimension formula then gives equality.  
\end{proof}
\section{Cumulants and the normalized Hessian determinant}\label{sec:ma}
\subsection{A fourth-cumulant inequality}

Let $Y$ have the tilted law \eqref{eq:n-3-2} at a positive admissible exponent, write $\Sigma=\operatorname{Cov}(Y)$, and let $\kappa_j$ denote its cumulants. The moment generating function is finite near zero. Differentiating its logarithm, as in Corollary~\ref{cor:exponential-family}, gives $\kappa_k=\alpha(-1)^{k+1}D^k\log p(e)$, and hence
\begin{align}
 \Sigma&=\alpha H_e,\label{eq:n-5-1}\\
 \kappa_3(v,v,\cdot) &= 2\alpha H_e(v^2,\cdot),\label{eq:n-5-2}\\
 \kappa_4(v^4)&=6\alpha S_4(v).\label{eq:n-5-3}
\end{align}
Therefore
\begin{equation}\label{eq:n-5-4}
\kappa_4(v^4)
\ge
\frac32\,
\kappa_3(v,v,\cdot)\,
\Sigma^{-1}\,
\kappa_3(v,v,\cdot).
\end{equation}
Indeed, the contraction on the right is $4\alpha\|v^2\|_e^2$, so the difference is exactly $6\alpha\mathcal D_e(v)$.

Equality in every direction, at one exponent and one tilt, is equivalent to the Jordan product form \eqref{eq:n-4-8}. Integrating over Gaussian directions gives the following scalar version.

\begin{proposition}[The fourth-moment defect]\label{prop:scalar-fourth-moment}
Let $Y$ have the tilted law \eqref{eq:n-3-2}, with $\alpha\in G(p)\setminus\{0\}$, and put $\Sigma=\operatorname{Cov}(Y)$ and $Z=\Sigma^{-1/2}(Y-\mathbb EY)$. Define $T_{ijk}=\mathbb EZ_iZ_jZ_k$ and $b_k=\sum_iT_{iik}$. Then
\begin{equation}\label{eq:n-5-5}
\mathfrak D(Z)
:=\mathbb E|Z|^4-n(n+2)-\|T\|_F^2-\frac12|b|^2\ge0.
\end{equation}
Equality at one exponent and one tilt holds if and only if $p$ has the Jordan product form \eqref{eq:n-4-8}.
\end{proposition}
\begin{proof}
Integrate the whitened version of \eqref{eq:n-5-4} against an independent standard Gaussian direction \(v\). Writing $\kappa^Z$ for the cumulants after whitening, the identity $\E v_iv_jv_kv_l=\delta_{ij}\delta_{kl}+\delta_{ik}\delta_{jl}+\delta_{il}\delta_{jk}$ gives \[
 \mathbb E_v\kappa_4^Z(v^4)=3\sum_{i,j}\kappa^Z_{iijj},
 \qquad
 \mathbb E_v|T(v,v,\cdot)|^2=2\|T\|_F^2+|b|^2.
\]
Since $\operatorname{Cov}(Z)=I$, we have $\sum_{i,j}\kappa^Z_{iijj}=\mathbb E|Z|^4-n(n+2)$; division by $3$ proves \eqref{eq:n-5-5}.

More precisely, let $G_e$ be the Gaussian vector with covariance $H_e^{-1}$, so that its coordinates in an $H_e$-orthonormal basis are independent standard Gaussians. A whitened direction $v$ corresponds to $\alpha^{-1/2}H_e^{-1/2}v$ in the original coordinates. Since $\mathcal D_e$ is quartic, integrating the exact difference $6\alpha\mathcal D_e$ and dividing by $3$ gives
\begin{equation}\label{eq:n-5-6}
\mathfrak D(Z)
=
\frac2\alpha\,\mathbb E\mathcal D_e(G_e).
\end{equation}
A nonnegative polynomial has zero integral against a full-support Gaussian only if it vanishes identically. Theorem~\ref{thm:n-4-1} gives the equality classification. \end{proof}

Define
\begin{equation}\label{eq:n-6-1}
\Psi(x)
=
\log\det H_x+\frac{2n}{m}\log p(x),
\qquad
\mathcal Lf=\operatorname{tr}(H_x^{-1}D^2f).
\end{equation}

The function $\Psi$ is homogeneous of degree zero, and a linear coordinate change adds a constant to it. The coefficients of the nondivergence operator $\mathcal L$ are the entries of the inverse Hessian.

\subsection{The exact defect identity}

\begin{proposition}\label{thm:n-6-1}
Let \(G_x\) be standard Gaussian in the \(H_x\) metric. Then
\begin{equation}\label{eq:n-6-2}
\mathcal L\Psi
=
2\,\mathbb E\mathcal D_x(G_x)
+
\frac12\|d\Psi\|_{H_x^{-1}}^2.
\end{equation}

Equivalently, for \(u=e^{-\Psi/2}\),
\begin{equation}\label{eq:n-6-3}
\mathcal Lu=-u\,\mathbb E\mathcal D_x(G_x)\le0.
\end{equation}
\end{proposition}
\begin{proof}
Both sides of \eqref{eq:n-6-2} will be expressed in terms of the same three contractions $A$, $N$ and $\zeta$ of the third and fourth derivatives at $x$. Fix $x$ and choose an $H_x$-orthonormal basis $v_i$. Extend these vectors as constant affine vector fields while differentiating; the basis is normalized only at the chosen point. Define
\[
c_{ijk}=\langle v_i\circ v_j,v_k\rangle_x,
\qquad
N=\sum_{i,j,k}c_{ijk}^2,
\]
\begin{equation}\label{eq:n-6-4}
\zeta=\sum_i v_i^2,
\qquad
A=\sum_{i,j}S_{4,x}(v_i,v_i,v_j,v_j),
\end{equation}
where \(S_{4,x}(\cdot,\cdot,\cdot,\cdot)\) is the symmetric polarization of the quartic.

The unit is $x$, so
\begin{equation}\label{eq:n-6-5}
\langle \zeta,x\rangle_x=n,
\qquad
\|x\|_x^2=m.
\end{equation}

In these coordinates,
\[
 D H[v_k]_{ij}=-2c_{ijk},\qquad
 D^2H[v_k,v_l]_{ij}=6S_{4,x}(v_i,v_j,v_k,v_l).
\]
Thus $D\log\det H[v_k]=\sum_iD H[v_k]_{ii}=-2\sum_i c_{iik}=-2\langle\zeta,v_k\rangle_x$. Since $D\log p(x)[v_k]=\langle x,v_k\rangle_x$, this gives
\begin{equation}\label{eq:n-6-6}
\nabla_{H_x}\Psi
=
-2\left(\zeta-\frac nm x\right).
\end{equation}

For its second derivative,
\[
D^2\log\det H[v,v]
=
\operatorname{tr}(H^{-1}D^2H[v,v])
-
\operatorname{tr}(H^{-1}DH[v]H^{-1}DH[v]).
\]
At the chosen point $H=I$; summing over $v=v_k$, we obtain
\[
 \sum_{i,k}6S_{4,x}(v_i,v_i,v_k,v_k)
 -\sum_{i,j,k}4c_{ijk}c_{jik}=6A-4N.
\]
The term $(2n/m)\log p$ contributes $(2n/m)\tr(H^{-1}(-H))=-2n^2/m$. Hence
\begin{equation}\label{eq:n-6-7}
\mathcal L\Psi
=
6A-4N-\frac{2n^2}{m}.
\end{equation}

The three fourth-moment pairings give $\E S_{4,x}(G_x)=3A$. They also give
\[
 \E\|G_x^2\|_x^2
 =\sum_{i,j,k,l,a}c_{ija}c_{kla}
  (\delta_{ij}\delta_{kl}+\delta_{ik}\delta_{jl}+\delta_{il}\delta_{jk})
 =\|\zeta\|_x^2+2N.
\]
Subtracting yields
\begin{equation}\label{eq:n-6-8}
\mathbb E\mathcal D_x(G_x)
=
3A-2N-\|\zeta\|_x^2.
\end{equation}
Also,
\[
\left\|\zeta-\frac nm x\right\|_x^2
=
\|\zeta\|_x^2-\frac{n^2}{m}.
\]
Substituting these expressions proves \eqref{eq:n-6-2}. For $u=e^{-\Psi/2}$ the chain rule yields
\[
 \mathcal Lu=u\left(-\frac12\mathcal L\Psi
             +\frac14\|d\Psi\|_{H^{-1}}^2\right),
\]
which is \eqref{eq:n-6-3}. \end{proof}

\subsection{The Hessian determinant equation}

Writing $n_a=\dim J_a$ and $r_a=\rank J_a$, call the Jordan product below \emph{balanced} when its weights satisfy
\begin{align}
 p&=c\prod_a\Delta_a^{k_a},\label{eq:n-6-9}\\
 \frac{k_ar_a}{n_a}&=\frac mn\qquad\text{for every }a.\label{eq:n-6-10}
\end{align}

\begin{theorem}\label{thm:n-6-2}
The following conditions are equivalent:

\begin{enumerate}
\item \(\Psi\) has a local maximum at some point of \(C\).
\item \(\Psi\) is constant on \(C\).
\item For some \(c_0>0\), $\det(-D^2\log p)=c_0p^{-2n/m}$. \item \(p\) has the balanced Jordan product form \eqref{eq:n-6-9}--\eqref{eq:n-6-10}.

\end{enumerate}
In particular, the irreducible solutions are precisely simple Euclidean Jordan determinants, up to invertible real linear transformations and positive factors.
\end{theorem}
\begin{proof}
At a local maximum \(e\), $d\Psi(e)=0$ and $\mathcal L\Psi(e)\le0$, so the nonnegative terms in \eqref{eq:n-6-2} must vanish. In particular, $\mathbb E\mathcal D_e(G_e)=0$. Thus \(\mathcal D_e\equiv0\), and Theorem~\ref{thm:n-4-1} gives \eqref{eq:n-6-9}.

For a simple Euclidean Jordan algebra, let $r$ be its rank and $N$ its dimension. The differential identities $D\log\Delta(x)[u]=\tr(x^{-1}\circ u)$ and $D(x^{-1})[u]=-P(x^{-1})u$ identify the Hessian of $-\log\Delta$ with $P(x^{-1})$ in trace coordinates \cite[Chapters~III--IV]{FK}. Choose a Jordan frame and write $x=\sum_i t_ic_i$. On the diagonal line $\R c_i$, the operator $P(x^{-1})$ has eigenvalue $t_i^{-2}$. On the Peirce space $J_{ij}$, $L_{x^{-1}}$ has eigenvalue $(t_i^{-1}+t_j^{-1})/2$ and $L_{x^{-2}}$ has eigenvalue $(t_i^{-2}+t_j^{-2})/2$; hence $2L_{x^{-1}}^2-L_{x^{-2}}$ has eigenvalue $(t_it_j)^{-1}$. The spaces $J_{ij}$ have the common dimension $d$, so $N=r+d r(r-1)/2$ and
\begin{equation}\label{eq:n-6-12}
\det(-D^2\log\Delta(x))
=
\prod_i t_i^{-2-d(r-1)}
=
\Delta(x)^{-2N/r}.
\end{equation}

Weights contribute only positive constants. Therefore $\det H_x = c_1\prod_a\Delta_a(x_a)^{-2n_a/r_a}$, and
\begin{equation}\label{eq:n-6-13}
\Psi(x)
=
c_2+
2\sum_a
\left(\frac nm k_a-\frac{n_a}{r_a}\right)
\log\Delta_a(x_a).
\end{equation}

Differentiate \eqref{eq:n-6-13} along a radial change in just the $a$th block. Since $D\log\Delta_a(x_a)[x_a]=r_a>0$, the equality $d\Psi(e)=0$ forces the coefficient of that block to vanish. This is \eqref{eq:n-6-10} for every $a$, and then $\Psi$ is constant on $C$.

Conversely, a balanced product makes every coefficient in \eqref{eq:n-6-13} zero. The equivalence of conditions 2 and 3 is immediate. Irreducibility leaves one factor with exponent one. \end{proof}

For \(m>1\), $D^2p=p(\tau_p\tau_p^{\mathsf T}-H)$, and $\tau_p^{\mathsf T}H^{-1}\tau_p=m$. By the determinant lemma,
\begin{equation}\label{eq:n-6-14}
\det D^2p
=
(-1)^{n+1}(m-1)p^n\det H.
\end{equation}
Thus Theorem~\ref{thm:n-6-2} also classifies the complete hyperbolic solutions of $\det D^2p=c\,p^{n(m-2)/m}$. The power is interpreted positively on \(C\); a global polynomial identity imposes the corresponding integrality conditions.

Fox relates the homogeneous Hessian equation to affine-sphere level sets and constructs solutions from relative invariants \cite{Fox}. 

Finally, for $\alpha\in G(p)\setminus\{0\}$, combining \eqref{eq:n-5-6}, \eqref{eq:n-6-2}, and \(\Sigma_x=\alpha H_x\) gives
\begin{equation}\label{eq:n-6-16}
\operatorname{tr}(\Sigma_x^{-1}D^2\Psi)
=
\mathfrak D(Z_{\alpha,x})
+
\frac12\|d\Psi\|_{\Sigma_x^{-1}}^2.
\end{equation}
This identifies the normalized geometric defect with the fourth-moment defect.
\Needspace{10\baselineskip}
\section{Polynomial multiplicative Legendre transforms}\label{sec:duality}
\subsection{The irreducible classification}

\begin{theorem}\label{thm:n-7-1}
Suppose $p$ is an irreducible complete real hyperbolic polynomial, positive on its full hyperbolicity cone $C$. Then \(p_*\) is a polynomial if and only if, after an invertible real linear change of coordinates and positive rescaling, \(p\) is the determinant of a simple Euclidean Jordan algebra.
\end{theorem}
\begin{proof}
The Legendre identities turn polynomiality of $p_*$ into a Hessian-determinant equation for $p$, which Theorem~\ref{thm:n-6-2} solves. Assume \(q=p_*\) is polynomial. Homogeneity makes \(q\) a homogeneous polynomial of degree \(m\). From the inverse-gradient identities,
\begin{equation}\label{eq:n-7-1}
q(\nabla p)=p^{m-1},
\qquad
\nabla q(\nabla p)=p^{m-2}x.
\end{equation}

Differentiate the second equation:
\[
D^2q(\nabla p)D^2p
=
p^{m-2}I
+
(m-2)p^{m-3}x(\nabla p)^{\mathsf T}.
\]
For \(m=2\), the second term is zero. Taking determinants and using
\(\langle x,\nabla p\rangle=mp\) gives
\begin{equation}\label{eq:n-7-2}
\det D^2q(\nabla p)\det D^2p
=
(m-1)^2p^{n(m-2)}.
\end{equation}

The real polynomial ring is a unique factorization domain. Since \(p\) is irreducible, \eqref{eq:n-7-2} forces
\begin{equation}\label{eq:n-7-3}
\det D^2p=c_1p^N,
\qquad
N=\frac{n(m-2)}m\in\mathbb Z_{\ge0}.
\end{equation}
The Hessian determinant is nonzero by \eqref{eq:n-6-14}, and homogeneity determines \(N\).

Equations \eqref{eq:n-6-14} and \eqref{eq:n-7-3} imply $\det H=c_2p^{-2n/m}$ with $c_2>0$. Theorem~\ref{thm:n-6-2} gives a balanced Jordan product. Irreducibility leaves one simple determinant with exponent one.

Conversely, for any Jordan product $p=c\prod_a\Delta_a^{k_a}$, standard trace coordinates give $\tau_p(x)_a=k_ax_a^{-1}$, and hence
\begin{equation}\label{eq:n-7-4}
p_*(y)
=
c^{-1}
\prod_a k_a^{-k_ar_a}\Delta_a(y_a)^{k_a}.
\end{equation}
This is polynomial. \end{proof}

\begin{remark}
The non-prehomogeneous Clifford quartics of Kogiso and Sato are absolutely irreducible \cite[Theorem~3.2(1),(3)]{KogisoSato} and have polynomial multiplicative Legendre transforms \cite[Theorem~2.14]{KogisoSato}. If one were complete hyperbolic, Theorem~\ref{thm:n-7-1} would identify it with a simple Euclidean Jordan determinant, which is a relative invariant of its structure group. Thus none of these counterexamples is complete hyperbolic.
\end{remark}

\begin{corollary}\label{thm:n-7-2}
An irreducible complete hyperbolic polynomial with polynomial multiplicative Legendre transform satisfies
\begin{equation}\label{eq:n-7-5}
m\mid2n.
\end{equation}

If an irreducible complete hyperbolic \(p\) is not a simple Jordan determinant, then \(\det D^2p\) has an irreducible factor not dividing \(p\).
\end{corollary}
\begin{proof}
The divisibility follows from \eqref{eq:n-7-3}. If all irreducible factors of \(\det D^2p\) divided \(p\), it would be a pure power of \(p\), and the same Monge--Amp\`ere argument would give the Jordan classification. \end{proof}

\subsection{Powers of the Legendre transform}

\begin{proposition}\label{thm:n-7-3}
For every complete hyperbolic \(p\), if \((p_*)^k\) is polynomial for some positive integer \(k\), then \(p_*\) is already polynomial. Moreover,
\begin{equation}\label{eq:n-7-6}
(p^k)_*=k^{-km}(p_*)^k.
\end{equation}
\end{proposition}
\begin{proof}
Let \(R=(p_*)^k\) be polynomial. Then $D\log p_*=(\nabla R)/(kR)$ is rational. By involutivity, \[
p_*(y)=\frac{1}{p\bigl(\nabla R(y)/(kR(y))\bigr)}.
\]
Thus \(p_*\) is rational, and its reduced denominator must be constant because $(p_*)^k$ is polynomial.

Finally,
\[
\tau_{p^k}=k\tau_p,
\qquad
\tau_p^{-1}(y/k)=k\tau_p^{-1}(y).
\]
Substitution proves \eqref{eq:n-7-6}. \end{proof}

\subsection{Reducible polynomials with visible boundary factors}

Write $p=c\prod_i h_i^{a_i}$ with distinct real irreducible factors.

Call \(h_i\) \emph{boundary-visible} if \(h_i=0\) contains a nonempty relatively open smooth hypersurface piece of \(\partial C\) on which every other \(h_j\) is nonzero.

\begin{theorem}\label{thm:n-7-4}
Let $p$ be complete and hyperbolic on its full cone $C$, and suppose every real irreducible factor is boundary-visible. Then the following are equivalent:

\begin{enumerate}
\item \(p_*\) is polynomial.
\item \(\det D^2p\) divides a positive integer power of \(p\), up to a nonzero constant.
\item Every entry of \(H_p^{-1}\) is polynomial.
\item \(p\) is a product of positive integer powers of independent simple Jordan determinants.

\end{enumerate}
In condition 3, the entries are necessarily homogeneous quadratics.
\end{theorem}
\begin{proof}
We prove $1\Rightarrow2\Rightarrow3\Rightarrow4\Rightarrow1$; boundary visibility is used only in $2\Rightarrow3$. Equation \eqref{eq:n-7-2} proves \(1\Rightarrow2\).

Let \(\mathsf M_x=H_x^{-1}\). The identity
\begin{equation}\label{eq:n-7-7}
\mathsf M_x=\frac{xx^{\mathsf T}}{m-1}
-p(x)(D^2p(x))^{-1}
\end{equation}
follows by multiplication by \(H_x\), using
$D^2p(x)x=(m-1)\nabla p(x)$.

Under condition 2, each reduced denominator has only irreducible factors of $p$. Suppose a reduced denominator contains $h_i$. Its smooth boundary patch contains a point $x_0$ where the numerator and all other denominator factors are nonzero. Indeed, locally write the patch as the graph of a smooth function over an open subset of $\R^{n-1}$. If a polynomial coprime to $h_i$ vanished on an open part of this graph, its nonzero resultant with $h_i$ over $\R(x_1,\ldots,x_{n-1})$ would vanish on an open set, a contradiction. Apply this to the product of the numerator and the other denominator factors.

The points $x_0+te$ belong to $C$ for $t>0$, and $h_i(x_0+te)$ tends to zero while the reduced numerator stays nonzero. The entry would therefore be unbounded as $t\downarrow0$. On the other hand, \eqref{eq:n-2-15} bounds the positive matrix $\mathsf M_x$ by $\lambda_1^e(x)^2H_e^{-1}$; continuity of the root bounds its norm on bounded subsets of $C$. This contradiction removes every denominator factor.

Homogeneity, $\mathsf M_{tx}=t^2\mathsf M_x$, then makes each entry quadratic. Thus \(2\Rightarrow3\).

To prove \(3\Rightarrow4\), differentiate in a fixed direction \(v\), using primes. From $\mathsf M\tau_p=x$ and $\tau_p'=-Hv$, one obtains $\mathsf M'\tau_p=2v$, $\mathsf M''\tau_p=\mathsf M'Hv$,
\begin{equation}\label{eq:n-7-8}
\mathsf M'''\tau_p=2\mathsf M''Hv+\mathsf M'H'v.
\end{equation}
The left-hand side is zero because $\mathsf M$ is quadratic. At the current point choose coordinates with $H=I$. Let \(L_vw=v\circ w\), and let \(T_v\) be the operator corresponding to the bilinear form $S_4(v,v,\cdot,\cdot)$. Then $H'=-2L_v$ and $H''=6T_v$, $\mathsf M'=2L_v$ and $\mathsf M''=8L_v^2-6T_v$.
These formulas follow from $(H^{-1})'=-H^{-1}H'H^{-1}$ and $(H^{-1})''=2H^{-1}H'H^{-1}H'H^{-1}-H^{-1}H''H^{-1}$. Substituting into \eqref{eq:n-7-8}, we obtain
\[
 0=2(8L_v^2-6T_v)v+(2L_v)(-2L_v)v,
\]
that is, $12(L_v^2-T_v)v=0$. Pairing with \(v\), $\|v^2\|^2=S_4(v)$. Thus \(\mathcal D_x\equiv0\), and Theorem~\ref{thm:n-4-1} proves condition 4.

Finally, \eqref{eq:n-7-4} proves \(4\Rightarrow1\). \end{proof}

\subsection{Dual defect reversal}
The conclusion that a hyperbolic barrier has a hyperbolic dual only in the self-scaled case is recorded in \cite[Section~3.4]{NT}; together with \cite{HauserGuler}, it gives the Jordan classification in the last sentence of the next proposition. The defect identity gives a direct proof in the present notation.

The derivative definition \eqref{eq:n-4-6} makes sense for the positive analytic function \(p_*\), even if it is not polynomial.

\begin{proposition}\label{thm:n-7-5}
For \(y=\tau_p(x)\) and \(w=H_xv\),
\begin{equation}\label{eq:n-7-9}
\mathcal D^{p_*}_y(w)=-\mathcal D^p_x(v).
\end{equation}

Consequently, \(p_*\) is a polynomial hyperbolic with respect to every point of \(C^\vee\) if and only if \(p\) is a Jordan product of the form \eqref{eq:n-4-8}.
\end{proposition}
\begin{proof}
The logarithmic gradient is an isometry of the two Hessian metrics that reverses the cubic form. We use this to compare their covariant derivatives, whose diagonal values are six times the respective defects. Let $F=-\log p$ and $G=-\log p_*$. For the usual convex Legendre transform \(F^*\), $G(y)=F^*(-y)+m$. The map \(\tau_p=-DF\) has differential \(-H_x\), and $D^2G(\tau_p(x))=H_x^{-1}$. Thus it is an isometry of the Hessian metrics.

Differentiating the inverse-Hessian identity gives $\tau_p^*(D^3G)=-D^3F$. An isometry preserves the Levi--Civita connection, so $\tau_p^*(\widehat\nabla^G D^3G) = -\widehat\nabla^F D^3F$. Evaluate four times on \(v\), and use \eqref{eq:n-4-37}. The minus sign in
\(D\tau_p(v)=-H_xv\) disappears in degree four, proving \eqref{eq:n-7-9}.

If \(p_*\) is also hyperbolic, both defects are nonnegative. Equation \eqref{eq:n-7-9} forces both to vanish. Apply Theorem~\ref{thm:n-4-1}. The converse follows from \eqref{eq:n-7-4}. \end{proof}

\section{Asymptotic profiles of the Riesz kernels}\label{sec:profiles}
\subsection{The Legendre transform in the saddle-point formula}

For \(y\in C^\vee\), put $x(y)=\tau_p^{-1}(y)$. At \(\alpha y\), the saddle of the inverse Laplace integral is \(x(y)\), independently of \(\alpha\). Euler's identity gives $\langle x(y),y\rangle=m$. Thus Theorem~\ref{thm:main} becomes
\begin{equation}\label{eq:n-8-1}
q_\alpha(\alpha y)
=
\frac{e^{m\alpha}}{(2\pi\alpha)^{n/2}}
\frac{p_*(y)^\alpha}{\sqrt{\det H_{x(y)}}}
(1+\varepsilon_\alpha(y)),
\end{equation}
where
$|\varepsilon_\alpha(y)|\le512n^2/\alpha$.

In particular,
\begin{equation}\label{eq:n-8-2}
\lim_{\alpha\to\infty}
\left(
\frac{q_\alpha(\alpha y)}{q_\alpha(\alpha y_0)}
\right)^{1/\alpha}
=
\frac{p_*(y)}{p_*(y_0)}.
\end{equation}

The exponential profile is thus the multiplicative Legendre transform. The next result uses the prefactor to distinguish the balanced Jordan products.

\subsection{Common-power asymptotics}

\begin{theorem}\label{thm:n-8-1}
The following conditions are equivalent:

\begin{enumerate}
\item There exist a fixed positive function \(\Phi:C^\vee\to(0,\infty)\), a sequence \(\alpha_j\to\infty\), and constants \(b_j,s_j>0\), such that for every fixed \(y\in C^\vee\),
\begin{equation}\label{eq:n-8-3}
\frac{q_{\alpha_j}(\alpha_jy)}
{b_j\Phi(y)^{s_j}}
\longrightarrow1.
\end{equation}
\item \(p\) is a balanced Jordan product.
\item \(p_*\) is polynomial and, for all sufficiently large real \(\alpha\),
\begin{equation}\label{eq:n-8-4}
q_\alpha(y)=c_\alpha p_*(y)^{\alpha-n/m}
\qquad(y\in C^\vee)
\end{equation}
for some \(c_\alpha>0\).

\end{enumerate}
Under condition 1, necessarily
\begin{equation}\label{eq:n-8-5}
\Phi=c_0p_*^a,
\qquad
s_j=\frac{\alpha_j-n/m}{a}+o(1)
\end{equation}
for some \(a,c_0>0\).
\end{theorem}
\begin{proof}
Equation~\eqref{eq:n-8-2} determines the ratios of $p_*$ from the leading exponential asymptotics. A common power profile must therefore be a power of $p_*$, and the prefactor in \eqref{eq:n-8-1} yields the Hessian-determinant equation. Assume condition 1 and fix \(y_0\in C^\vee\). By kernel homogeneity,
\[
\frac{q_{\alpha_j}(2\alpha_jy_0)}
{q_{\alpha_j}(\alpha_jy_0)}
=
2^{m\alpha_j-n}.
\]
Let $d=\log(\Phi(2y_0)/\Phi(y_0))$. The relative limits imply $s_jd=(m\alpha_j-n)\log2+o(1)$. Since \(s_j>0\), we must have \(d>0\). Set $a=d/(m\log2)$. Then the second formula in \eqref{eq:n-8-5} follows.

Compare \eqref{eq:n-8-3} at \(y\) and \(y_0\), take \(\alpha_j\)-th roots, and use \eqref{eq:n-8-2}. We obtain
\[
\frac{\Phi(y)}{\Phi(y_0)}
=
\left(\frac{p_*(y)}{p_*(y_0)}\right)^a.
\]
Thus $\Phi=c_0p_*^a$, and absorbing constants into $b_j$ reduces condition 1 to
\begin{equation}\label{eq:n-8-7}
\frac{q_{\alpha_j}(\alpha_jy)}
{\widetilde b_jp_*(y)^{\alpha_j-n/m}}
\longrightarrow1.
\end{equation}
The $o(1)$ correction in the exponent contributes $1+o(1)$ at each fixed $y$. Comparing two points in \eqref{eq:n-8-1}, now at the level of relative error, shows that \eqref{eq:n-8-7} forces the function
\[
\mathcal A_p(y)=\frac{p_*(y)^{n/m}}{\sqrt{\det H_{x(y)}}}
\]
to be constant.

Since \(p_*(y)=p(x(y))^{-1}\), this is exactly $\det H_x=\text{constant}\cdot p(x)^{-2n/m}$. Theorem~\ref{thm:n-6-2} proves condition 2.

Conversely, suppose $p=c\prod_a\Delta_a^{k_a}$ is balanced. The classical cone Gamma integral gives, for sufficiently large \(\alpha\),
\begin{equation}\label{eq:n-8-9}
q_\alpha(y)
=
\frac{c^{-\alpha}}
{\prod_a\Gamma_{\Omega_a}(k_a\alpha)}
\prod_a
\Delta_a(y_a)^{k_a\alpha-n_a/r_a}.
\end{equation}
This is the cone gamma integral in trace-normalized Lebesgue measure \cite[Chapter~VII]{FK}. 

Balance says $n_a/r_a=k_an/m$. Together with \eqref{eq:n-7-4}, this turns \eqref{eq:n-8-9} into \eqref{eq:n-8-4}. In trace-normalized coordinates,
\[
c_\alpha
=
\frac{
c^{-n/m}
\prod_a k_a^{k_ar_a(\alpha-n/m)}
}{
\prod_a\Gamma_{\Omega_a}(k_a\alpha)
}.
\]
A change of coordinates affects only this constant. Condition 3 implies condition 1 by homogeneity, with zero error. \end{proof}

\subsection{The first correction term}

\begin{proposition}\label{thm:n-8-2}
Fix \(e\in C\) and put $y_\alpha=\alpha\tau_p(e)$. Let \(N,\zeta\) be defined as in \eqref{eq:n-6-4}, and let $\mu_e=\mathbb E\mathcal D_e(G_e)$. Then
\begin{equation}\label{eq:n-8-10}
\frac{q_\alpha(y_\alpha)}{Q_\alpha(y_\alpha)}
=
1+\frac{a_1(e)}{\alpha}
+O_{p,e}(\alpha^{-2}),
\end{equation}
where
\begin{equation}\label{eq:n-8-11}
a_1(e)=\frac14\mu_e+\frac16N-\frac14\|\zeta\|_e^2.
\end{equation}

Consequently, $a_1(e)\ge\frac16N-\frac14\|\zeta\|_e^2$, and equality at a single point characterizes the Jordan product form.
\end{proposition}
\begin{proof}
We expand the normalized integrand to order $1/\alpha$, integrate the resulting polynomial against the Gaussian, and control the remainder on three radial regions. Use the $H_e$-orthonormal coordinates of \eqref{eq:a-normalization}. In this proof put $b_j(z)=\lambda_j^e(H_e^{-1/2}z)$ and $S_k(z)=\sum_jb_j(z)^k$, and write $W_\alpha$ for the integrand in \eqref{eq:a-W}. Expanding its logarithmic modulus and the cosine of its phase, we obtain
\begin{equation}\label{eq:n-8-13}
\operatorname{Re}W_\alpha(z)
=
e^{-\|z\|^2/2}
\left[
1+\frac1\alpha
\left(\frac{S_4(z)}4-\frac{S_3(z)^2}{18}\right)
\right]
+\text{remainder}.
\end{equation}

To estimate the integrated remainder, put $\rho=\|z\|$. Then $\sum_jb_j(z)^2=\rho^2$ and $\sum_j|b_j(z)|^k\le\rho^k$ for $k\ge2$. Write the normalized integrand from \eqref{eq:a-W} as
\[
 W_\alpha=e^{-\rho^2/2+R_\alpha+i\phi},\qquad
 r_\alpha=\frac{S_4(z)}{4\alpha},\qquad
 \vartheta_\alpha=\frac{S_3(z)}{3\sqrt\alpha}.
\]
On $\rho\le\alpha^{1/20}$ and for $\alpha\ge16$, the integral remainders for $\log(1+s)$ and $\arctan t$ give
\begin{equation}\label{eq:profile-local-remainders}
 0\le R_\alpha\le r_\alpha,\quad |R_\alpha-r_\alpha|\le\frac{\rho^6}{6\alpha^2},\quad
 |\phi-\vartheta_\alpha|\le\frac{\rho^5}{5\alpha^{3/2}},\quad
 |\phi|\le\frac{\rho^3}{3\sqrt\alpha}.
\end{equation}
For example, integrate the remainder in $1/(1+s)=1-s+s^2/(1+s)$ to obtain the logarithmic estimate, and in $1/(1+t^2)=1-t^2+t^4/(1+t^2)$ to obtain the phase estimate. Taylor's theorem then gives
\begin{align}\label{eq:profile-cos-exp-remainders}
 \left|\cos\phi-1+\frac{\vartheta_\alpha^2}{2}\right|
 &\le\frac1{\alpha^2}\left(\frac{\rho^8}{15}+\frac{\rho^{12}}{1944}\right),\\
 |e^{R_\alpha}-1-r_\alpha|&\le\frac1{\alpha^2}
       \left(\frac{\rho^6}{6}+\frac{\rho^8e^{R_\alpha}}{32}\right).
\end{align}
The first inequality uses $|\phi^2-\vartheta_\alpha^2|/2\le\rho^8/(15\alpha^2)$ and $|\cos\phi-1+\phi^2/2|\le\phi^4/24$. The second uses $|e^{R_\alpha}-1-R_\alpha|\le R_\alpha^2e^{R_\alpha}/2$. Since $r_\alpha\le1/4$ and $r_\alpha\le\rho^2/12$ on this region, multiplication of the two expansions gives
\begin{equation}\label{eq:profile-integrable-remainder}
 \left|\Re W_\alpha-e^{-\rho^2/2}
       \left(1+\frac{S_4(z)}{4\alpha}
                  -\frac{S_3(z)^2}{18\alpha}\right)\right|
 \le \frac{C}{\alpha^2}
      (\rho^6+\rho^8+\rho^{10}+\rho^{12})e^{-\rho^2/3}
\end{equation}
with an absolute constant $C$. In particular its integral is $O_n(\alpha^{-2})$.

For $\alpha^{1/20}<\rho\le\sqrt\alpha$, the envelope from \eqref{eq:a-envelope} satisfies $K_\alpha(z)\le e^{-\rho^2/4}$, because $\log(1+s)\ge s/2$ for $0\le s\le1$. The comparison terms are a fixed polynomial in $\rho$ times a Gaussian. Their integrals beyond $\alpha^{1/20}$ are $O_n(\alpha^{-L})$ for every fixed $L$: absorb the polynomial into half of the Gaussian and use $e^{-c\alpha^{1/10}}=O(\alpha^{-L})$. Finally, for $\alpha>2n$, polar coordinates and $1+t^2\ge2t$ for $t\ge1$ give
\begin{equation}\label{eq:profile-outer-tail}
 \int_{\rho\ge\sqrt\alpha}K_\alpha(z)\,dz
 \le |S^{n-1}|\alpha^{n/2}
          \frac{2^{-\alpha/2}}{\alpha/2-n}.
\end{equation}
The corresponding Gaussian tail also decays exponentially. These three regions prove the $O(\alpha^{-2})$ integrated remainder.

For a standard Euclidean Gaussian $Z$, the three fourth-order pairings give $\E S_4(Z)=3A$. Among the fifteen pairings of the six factors in $S_3(Z)^2$, six connect all three indices across the two cubic tensors and nine have one internal pair in each tensor. Therefore
\begin{equation}\label{eq:profile-wick}
 \E S_4(Z)=3A,\qquad \E S_3(Z)^2=6N+9\|\zeta\|_e^2.
\end{equation}
Thus $a_1(e)=3A/4-N/3-\|\zeta\|_e^2/2$. Using $\mu_e=3A-2N-\|\zeta\|_e^2$ proves \eqref{eq:n-8-11}. Nonnegativity and Theorem~\ref{thm:n-4-1} give the equality statement. \end{proof}

\section{Examples}\label{sec:examples}
\subsection{Lorentz determinants}

For $p(t,z)=t^2-\|z\|^2$ and $e=(1,0)$, one has \(H_e=2I\) and \(\mathcal D_e=0\).

Take $v_i=e_i/\sqrt2$. The nonzero cubic structure constants are $c_{000}=1/\sqrt2$ and $c_{0ii}=c_{i0i}=c_{ii0}=1/\sqrt2$ for $1\le i<n$. Thus $N=(3n-2)/2$ and $\|\zeta\|_e^2=n^2/2$. Equation \eqref{eq:n-8-11} gives
\begin{equation}\label{eq:n-9-1}
\frac{q_\alpha(2\alpha e)}{Q_\alpha(2\alpha e)}
=
1-\frac{3n^2-6n+4}{24\alpha}
+O_n(\alpha^{-2}).
\end{equation}
This recovers the Stirling coefficient in Section~\ref{subsec:lorentz}. For positive exponents, complete monotonicity holds exactly when $\alpha\ge(n-2)/2$.

\subsection{Third-order information is insufficient}

Consider the polynomials $p_0(t,x,y)=(t^2-x^2-y^2)^2$ and $p_1(t,x,y)=t^2(t^2-2x^2-2y^2)$. At $e=(1,0,0)$, both have $H_e=4I$ and their logarithmic derivatives agree through order three. Indeed, with \(u=(x^2+y^2)/t^2\), $\log p_0=4\log t-2u-u^2+O(u^3)$, $\log p_1=4\log t-2u-2u^2+O(u^3)$.
Nevertheless,
\begin{equation}\label{eq:n-9-2}
\mathcal D_e^{p_0}(a,b,c)=0,
\qquad
\mathcal D_e^{p_1}(a,b,c)=4(b^2+c^2)^2.
\end{equation}
Thus even a Jordan multiplication obtained from the third derivative does not identify the original polynomial without fourth-order compatibility.

For rational \(p,e\), the test can be made using one rational scalar. Let
\[
\Delta_H
=
\sum_{i,j}(H_e^{-1})_{ij}
\frac{\partial^2}{\partial v_i\partial v_j}.
\]
For the homogeneous quartic \(\mathcal D_e\),
\begin{equation}\label{eq:n-9-3}
\mathbb E\mathcal D_e(G_e)=\frac18\Delta_H^2\mathcal D_e.
\end{equation}
The fourth Gaussian moment formula gives \eqref{eq:n-9-3}: the three contractions of a symmetric fourth-order tensor coincide, while applying $\Delta_H$ twice contributes the factor $24$. For complete hyperbolic $p$, vanishing of this rational scalar is equivalent to the Jordan product form.

\subsection{The dual Vinberg cone}
We use the naming convention of Ishi \cite[Section~5.2]{Ishi}: the sparse positive definite matrix cone below is the dual of the five-dimensional Vinberg cone \cite{Vinberg}. Multiplicative Legendre transforms of generalized powers on homogeneous cones are studied in \cite[Theorem~5 and Proposition~1]{Ishi2016}.

Let
\begin{equation}\label{eq:n-9-4}
f(a,b,c,x,y)
=
abc-bx^2-ay^2
=
\det
\begin{pmatrix}
a&0&x\\
0&b&y\\
x&y&c
\end{pmatrix}.
\end{equation}
Its complete hyperbolicity cone is $C_f=\{a>0,\ b>0,\ c>x^2/a+y^2/b\}$. The determinant is hyperbolic at the identity matrix, and the displayed cone is its positive definite component by the Schur complement. At the identity its logarithmic Hessian on this five-dimensional matrix space is the positive trace form, so the cone is pointed. Irreducibility follows from Gauss's lemma: as a polynomial in $c$ it is primitive and linear over $\R(a,b,x,y)$, because $ab$ and $bx^2+ay^2$ are coprime.

Use the dual pairing $aA+bB+cC+xX+yY$, and put $\widehat A=A-X^2/(4C)$ and $\widehat B=B-Y^2/(4C)$. Write $c=s+x^2/a+y^2/b$ with $s>0$ in the primal cone. Minimizing the pairing in $x,y$, we obtain
\[
 \begin{aligned}
 aA+bB+cC+xX+yY
 &=Cs+a\widehat A+b\widehat B\\
 &\quad+\frac C a\left(x+\frac{aX}{2C}\right)^2
  +\frac C b\left(y+\frac{bY}{2C}\right)^2.
\end{aligned}
\]
Consequently the dual interior is
$D_f=\{C>0,\ \widehat A>0,\ \widehat B>0\}$.

The inverse logarithmic gradient is
\begin{align*}
 a&=\widehat A^{-1}, & b&=\widehat B^{-1},\\
 x&=-\frac{aX}{2C}, & y&=-\frac{bY}{2C},
 & c&=\frac1C+\frac{x^2}{a}+\frac{y^2}{b}.
\end{align*}
Therefore
\begin{equation}\label{eq:n-9-5}
f_*
=
C\widehat A\widehat B
=
\frac{(AC-X^2/4)(BC-Y^2/4)}{C}.
\end{equation}
The denominator does not cancel. For $\alpha>1/2$, its positive Riesz density is
\begin{equation}\label{eq:n-9-6}
q_\alpha
=
\frac{
C^{\alpha-2}
\widehat A^{\alpha-3/2}
\widehat B^{\alpha-3/2}
}{
4\pi\Gamma(\alpha)\Gamma(\alpha-1/2)^2
}
\end{equation}
on \(D_f\).

To verify \eqref{eq:n-9-6}, the change from $(A,B,C,X,Y)$ to $(\widehat A,\widehat B,C,X,Y)$ has Jacobian one. All integrands are nonnegative, so Tonelli's theorem permits us to integrate in \(\widehat A,\widehat B\), then in \(X,Y\) by Gaussian integration, and finally in \(C\). The resulting factors are $\Gamma(\alpha-1/2)^2 a^{-\alpha+1/2}b^{-\alpha+1/2}$, \[
\frac{4\pi C}{\sqrt{ab}}
\exp\!\left(C\left(\frac{x^2}{a}+\frac{y^2}{b}\right)\right),
\]
and the Gamma integral with shape \(\alpha\) and rate
\(c-x^2/a-y^2/b\). Their product is \(f^{-\alpha}\).
Thus
\begin{equation}\label{eq:n-9-7}
q_\alpha
=
\text{constant}_\alpha\,
f_*^{\alpha-3/2}C^{-1/2},
\end{equation}
so the ratio $q_\alpha/f_*^{\alpha-5/3}$ is a constant multiple of $f_*^{1/6}C^{-1/2}$ and is not constant, in accordance with Theorem~\ref{thm:n-8-1}.

\subsection{A factor invisible on the boundary}

Let $P=abf$. The full cone remains $C_f$. On its closure, $a=0$ forces $x=0$ and $b=0$ forces $y=0$, so these linear factors meet the boundary in ambient codimension at least two and are not boundary-visible. Differentiating $\log P=2\log a+2\log b+\log(c-x^2/a-y^2/b)$ and solving its gradient equations gives
\begin{align}
 P_*&= \frac{(AC-X^2/4)^2(BC-Y^2/4)^2}{16C^3},\label{eq:n-9-8}\\
 \det D^2P&=64a^6b^6f,\label{eq:n-9-9}
\end{align}
and
\begin{equation}\label{eq:n-9-10}
(H_P^{-1})_{cc}
=
c^2-\frac{x^4}{2a^2}
-\frac{y^4}{2b^2}
-\frac{2x^2y^2}{ab}.
\end{equation}
Thus all factors of the ordinary Hessian determinant come from \(P\), while the inverse logarithmic Hessian still has poles. This disproves the unrestricted implication from Hessian divisibility to polynomial inverse Hessian.

We do not know whether condition~1 of Theorem~\ref{thm:n-7-4} implies condition~4 without the boundary-visibility hypothesis; in the example above both fail, since $P_*$ in \eqref{eq:n-9-8} is not a polynomial.

\appendix
\section{Factor obstructions to polynomial duality}\label{sec:factor-obstructions}
The multiplicative Legendre identities also impose restrictions on individual factors over $\C$. We record two such restrictions.

Let \(P\) be homogeneous of degree \(M\) on a complex vector space \(V\), with generically nonsingular logarithmic Hessian, and suppose its multiplicative Legendre transform \(Q\) is polynomial.

For an irreducible homogeneous \(h\), the projective dual of $X=\{h=0\}\subset\mathbb P(V)$ is the closure of the tangent hyperplanes to the smooth locus of \(X\). We use projective biduality and its correspondence between smooth tangent points \cite[Theorem~1.7, arXiv version]{Tevelev}. 

\begin{theorem}\label{thm:n-A-1}
Assume \(\dim V\ge3\). Suppose an irreducible factor \(h\) of \(P\) has degree \(d\ge2\), and the projective dual of \(\{h=0\}\) is a hypersurface in \(\mathbb P(V^*)\). Then \(h\) is a nondegenerate quadratic form on \(V\), and $P=ch^a$ with $c\ne0$ and $a\in\mathbb Z_{>0}$. \end{theorem}
\begin{proof}
Write $P=h^ar$ with $h\nmid r$. On \(r\ne0\), define
\begin{equation}\label{eq:n-A-2}
Z(x)=h(x)D\log P(x)
=
a\,Dh(x)+h(x)\frac{Dr(x)}{r(x)}.
\end{equation}
It is regular across \(h=0\).

Since \(Q\) has degree \(M\),
\begin{equation}\label{eq:n-A-3}
Q(Z(x))=\frac{h(x)^{M-a}}{r(x)},
\qquad
\langle x,Z(x)\rangle=Mh(x).
\end{equation}

Let \(h^\vee\) define the projective dual, and let
\(\delta=\deg h^\vee\). Choose a sufficiently general nonzero smooth point \(x_0\) of \(h=0\), with \(r(x_0)\ne0\), whose Gauss image is smooth on \(h^\vee=0\).

Put $n=\dim V$ and $z_0=aDh(x_0)$. Since both projective varieties are hypersurfaces, biduality makes their Gauss maps inverse on dense open sets of smooth points. Their differentials therefore have rank $n-2$ there. On $T_{x_0}\{h=0\}$ we have $DZ=aD^2h$, because the remaining terms contain either $h$ or $Dh[v]$. The radial derivative satisfies $D(Dh)_{x_0}x_0=(d-1)Dh(x_0)\ne0$. Thus \(DZ_{x_0}\) restricted to \(T_{x_0}\{h=0\}\) has rank \(n-1\), with image equal to the tangent space of the affine dual hypersurface.

Differentiating the second identity in \eqref{eq:n-A-3} at \(x_0\) gives $\langle x_0,DZ_{x_0}v\rangle = (M-a)Dh(x_0)[v]$. The covector \(x_0\) annihilates the tangent space of the dual hypersurface. Since $M\ge ad\ge2a$, a transverse vector is mapped outside that tangent space. Therefore \(DZ_{x_0}\) is invertible.

The complex inverse function theorem now makes $Z$ a local biholomorphism. Choose $h^\vee$ reduced. Because $Z$ carries the smooth hypersurface $h=0$ to $h^\vee=0$ with invertible differential, locally $h^\vee(Z(x))=u(x)h(x)$ with $u(x_0)\ne0$. Factor $Q=(h^\vee)^bq$ and choose the point generally enough that $q(z_0)\ne0$. Then $Q(Z(x))=h(x)^b$ times a holomorphic unit. Comparing with \eqref{eq:n-A-3}, where $r(x_0)\ne0$, gives $b=M-a$; equivalently, $\operatorname{ord}_{h^\vee}Q=M-a$. Hence $\delta(M-a)\le M$ and $ad\le M$. By biduality, \(\delta\ge2\): if the dual were a hyperplane, its dual would be a point, not the nonlinear hypersurface \(X\). Therefore $2a\le ad\le M\le2a$. All inequalities are equalities: $d=\delta=2$, $M=2a$, and $\deg r=0$. If the quadratic form had rank $s<n$, its projective Gauss image would have dimension $s-2<n-2$, so it could not be a hypersurface in the full dual projective space. Hence it is nondegenerate. \end{proof}

In particular, when $\dim V\ge3$ and $P$ has an irreducible factor of degree at least three whose projective dual is a hypersurface, $P_*$ is not a polynomial, whatever the remaining factors are.

\begin{proposition}[Interior linear factors]\label{thm:n-A-2}
Return to complete real hyperbolic \(p\). Suppose $p=\ell^ar$ with $\ell\nmid r$ and $\deg r>0$, and the linear functional \(\ell\) lies in \(C^\vee=\operatorname{int}K\). Then \(p_*\) is not polynomial.
\end{proposition}
\begin{proof}
If $p_*$ were polynomial, clearing denominators in the Legendre identity on $C$ would give a polynomial identity on the whole real vector space. Since $\ell\nmid r$, the restriction of $r$ to the real hyperplane $\ell=0$ is a nonzero polynomial, so it is nonzero at some real point $x_0$ of that hyperplane. Evaluate \eqref{eq:n-A-3} there with $h=\ell$. Then $Z(x_0)=a\ell$ and $M-a=\deg r>0$, giving $p_*(a\ell)=0$. But \(a\ell\in C^\vee\), where \(p_*>0\), a contradiction. \end{proof}

For a concrete boundary comparison, let
\[
p(x,y,z)=(xz-y^2)^a x^b,
\qquad a\ge1,\quad b\ge0,
\]
and set \(\Delta=AC-B^2/4\). Then
\begin{equation}\label{eq:n-A-6}
p_*(A,B,C)
=
\frac{\Delta^{a+b}}
{a^a(a+b)^{a+b}C^b}.
\end{equation}
Its inverse logarithmic gradient is
\[
x=\frac{(a+b)C}{\Delta},\qquad
y=-\frac{(a+b)B}{2\Delta},\qquad
z=\frac aC+\frac{(a+b)B^2}{4C\Delta}.
\]
These give $xz-y^2=a(a+b)/\Delta$, and verify \eqref{eq:n-A-6}. When \(b>0\), the denominator cannot cancel because
\(\Delta|_{C=0}=-B^2/4\).

\Needspace{10\baselineskip}
\section*{Use of AI tools}

This paper was produced with substantial use of AI language models. Candidate arguments, including the saddle-point proof of Theorem~\ref{thm:main}, the results on admissible exponents, and the Jordan rigidity and duality theory, were generated by GPT-6 Pro over several rounds directed by the author, who set the targets, identified the Jordan reconstruction and its role, rejected unsupported or overstated claims, and decided what to keep. The same model produced self-checks of its own output, and literature searches were assisted by GPT-5.6 Sol. The author determined the organization of the paper, including the merger of the positivity and rigidity results into a single work, and directed the rewriting of the manuscript.

AI-assisted step-by-step checking of the proofs and numerical constants, literature and priority cross-checks (including the identification of the question of Scott and Sokal answered here), and critical assessment of the claims, their framing and the exposition were carried out by Claude (Fable 5.1 and Opus 5.5) and by GPT Astra within Codex. The latter also assembled the LaTeX manuscript and wrote and ran the exact-arithmetic verification scripts.

All AI-generated text was reviewed and approved by the author. The author has checked the statements and proofs, is responsible for the mathematics, the citations and the final text, and alone made the decision to submit.

\section*{Data availability}
The LaTeX source and Python scripts for the scalar and exact-algebra calculations are included in the accompanying supplementary material.

\end{document}